\documentclass[reqno]{amsart}
\usepackage[foot]{amsaddr}
\usepackage[bottom=3cm, top=3cm, left=3cm, right=3cm]{geometry}
\usepackage{amsmath}
\usepackage{amsthm}
\usepackage{amssymb}
\usepackage{etoolbox}
\usepackage{mathtools}
\usepackage{cases}
\usepackage{empheq}
\usepackage[shortlabels,inline]{enumitem}
\setlist[itemize]{noitemsep,nolistsep}
\setlist[enumerate]{noitemsep,nolistsep}

\usepackage[dvipsnames]{xcolor}
\usepackage{graphicx}
\definecolor{tabblue}{rgb}{0.12156862745098039, 0.4666666666666667, 0.7058823529411765}
\definecolor{taborange}{rgb}{1.0, 0.4980392156862745, 0.054901960784313725}
\definecolor{tabgreen}{rgb}{0.17254901960784313, 0.6274509803921569, 0.17254901960784313}
\definecolor{tabred}{rgb}{0.8392156862745098, 0.15294117647058825, 0.1568627450980392}
\definecolor{tabpurple}{rgb}{0.5803921568627451, 0.403921568627451, 0.7411764705882353}
\definecolor{cblue}{HTML}{2b50aa}
\colorlet{citecolor}{tabgreen}
\colorlet{linkcolor}{cblue!90}
\colorlet{urlcolor}{tabred}

\usepackage{tikz}
\usepackage{tikz-cd}
\usepackage{wrapfig}

\usepackage{hyperref}
\usepackage{url}

\hypersetup{
  breaklinks=true,
  colorlinks=true,
  linkcolor=linkcolor,
  urlcolor=urlcolor,
  citecolor=citecolor,
  bookmarksdepth=3,
  pdftitle={Learning Monge maps with constrained drifting models},
  pdfauthor={Théo Dumont, Théo Lacombe, François-Xavier Vialard},
  }

\DeclareUnicodeCharacter{0308}{o}
  
\usepackage[backref=true, backrefstyle=three, hyperref=true, maxcitenames=3, maxbibnames=50, style=alphabetic, backend=bibtex]{biblatex}
\cslet{blx@noerroretextools}\empty  % autonum and biblatex https://github.com/plk/biblatex/issues/669
\usepackage[noabbrev,capitalize,nameinlink]{cleveref}
\crefname{equation}{}{}
\Crefname{equation}{Eq.}{}

\usepackage{autonum}
\usepackage{thmtools}
\usepackage{algorithm}
\usepackage{algpseudocodex}
\usepackage{subcaption}

\makeatletter
\renewcommand{\tocsection}[3]{%
  \indentlabel{\@ifnotempty{#2}{\bfseries\ignorespaces#1 #2\quad}}\bfseries#3}
\renewcommand{\tocsubsection}[3]{%
  \indentlabel{\@ifnotempty{#2}{\ignorespaces#1 #2\quad}}#3}
\def\l@subsection{\@tocline{2}{0pt}{2.5pc}{5pc}{}}
\makeatother

\allowdisplaybreaks

\newcounter{counter}
\numberwithin{counter}{section}
\numberwithin{equation}{section}
\newtheorem{theorem}[counter]{Theorem}
\newtheorem*{theorem*}{Theorem}

\newtheorem{lemma}[counter]{Lemma}
\newtheorem{innernblemma}{Lemma}
\newenvironment{nblemma}[1]{\renewcommand\theinnernblemma{#1}\innernblemma}{\endinnernblemma}
\newtheorem{proposition}[counter]{Proposition}
\newtheorem*{proposition*}{Proposition}

\newtheorem{corollary}[counter]{Corollary}
\theoremstyle{definition}
\newtheorem{remark}[counter]{Remark}
\newtheorem{example}[counter]{Example}
\newtheorem{definition}[counter]{Definition}

\NewCommandCopy{\proofqedsymbol}{\qedsymbol}% save the default
\AtBeginEnvironment{proof}{\renewcommand{\qedsymbol}{\proofqedsymbol}}
\AtBeginEnvironment{example}{\pushQED{\qed}\renewcommand{\qedsymbol}{$\diamondsuit$}}
\AtEndEnvironment{example}{\popQED\endexample}
\AtBeginEnvironment{remark}{\pushQED{\qed}\renewcommand{\qedsymbol}{$\triangle$}}
\AtEndEnvironment{remark}{\popQED\endexample}

\makeatletter
\newcommand\definealphabetloop[3]{%
  \ifx\relax#3\expandafter\@gobble\else\expandafter\@firstofone\fi
  {\expandafter\providecommand\expandafter*\csname#1#3\endcsname{#2{#3}}%
   \definealphabetloop{#1}{#2}}%
}%
\newcommand\definealphabet[2]{%
  \definealphabetloop{#1}{#2}abcdefghijklmnopqrstuvwxyzABCDEFGHIJKLMNOPQRSTUVWXYZ\relax
}%
\definealphabet{b}{\mathbb}%
\definealphabet{c}{\mathcal}%
\definealphabet{f}{\mathfrak}%
\definealphabet{r}{\mathrm}%
\definealphabet{s}{\mathscr}%
\definealphabet{t}{\mathtt}%
\definealphabet{bf}{\mathbf}%
\makeatother

\renewcommand{\epsilon}{\varepsilon}
\newcommand{\eps}{\varepsilon}
\renewcommand{\rho}{\varrho}
\usepackage{scalerel}
\def\smallzero{{\scaleto{0}{2.5pt}}}
\def\smallone{{\scaleto{1}{2.5pt}}}
\def\smalltwo{{\scaleto{2}{2.5pt}}}
\newcommand{\W}{\operatorname{W}}
\newcommand{\PRd}{\smash{\cP_2(\bR^d)}}
\newcommand{\PRdRd}{\smash{\cP_2(\bR^d\times\bR^d)}}
\newcommand{\PRdac}{\smash{\cP_{2}^\text{ac}(\bR^d)}}

\newcommand{\id}{{\operatorname{id}}}
\newcommand{\K}{{\smash{K_{\rho_\smallzero}}}}
\newcommand{\Lrhoz}{\smash{L^2_{\rho_\smallzero}}}
\newcommand{\LrhozRd}{\smash{L^2_{\rho_\smallzero}(\bR^d,\bR^d)}}
\newcommand{\LrhoRd}{\smash{L^2_{\rho}(\bR^d,\bR^d)}}

\newcommand{\Lrho}{\smash{L^2_{\rho}}}
\newcommand{\dHrhoz}{\smash{\dot H^1_{\rho_\smallzero}}}
\newcommand{\dHrhoRd}{\smash{\dot H^1_{\rho}(\bR^d,\bR)}}
\newcommand{\dHrhozRd}{\smash{\dot H^1_{\rho_\smallzero}(\bR^d,\bR)}}
\newcommand{\optmap}[2]{\smash{T_{#1}^{#2}}}

\newcommand{\Tan}{\operatorname{Tan}}
\newcommand{\Nor}{\operatorname{Nor}}

\newcommand{\nablaa}{\nabla_{\!{\scriptscriptstyle \W}}}

\newcommand{\grad}{\operatorname{grad}}

\DeclareMathOperator*{\argmin}{arg\,min}
\newcommand\dd{\mathop{}\!\mathrm d}
\newcommand{\proj}{\operatorname{proj}}
\renewcommand{\div}{\operatorname{div}}
\newcommand{\midd}{\,|\,}

\usepackage{fancyhdr}
\title{Learning Monge maps with constrained drifting models}
\author[Dumont]{Théo Dumont$^\dagger$}
\author[Lacombe]{Théo Lacombe$^\dagger$}
\author[Vialard]{François--Xavier Vialard$^\dagger$}
\address{\textnormal{$^\dagger$Laboratoire d'Informatique Gaspard Monge, Université Gustave Eiffel, CNRS, F-77454 Marne-la-Vallée, France.}}
\email{\{\href{mailto:theo.dumont@univ-eiffel.fr}{theo.dumont},\href{mailto:theo.lacombe@univ-eiffel.fr}{theo.lacombe},\href{mailto:francois-xavier.vialard@univ-eiffel.fr}{francois-xavier.vialard}\}@univ-eiffel.fr}
\begin{document}

\begin{abstract}
    We study the estimation of optimal transport (OT) maps between an arbitrary source probability measure and a log-concave target probability measure. Our contributions are twofold.
    First, we propose a new evolution equation in the set of transport maps. It can be seen as the gradient flow of a lift of some user-chosen divergence (e.g., the KL divergence, or relative entropy) to the space of transport maps, \emph{constrained} to the convex set of \emph{optimal} transport maps. We prove the existence of long-time solutions to this flow as well as its convergence toward the OT map as time goes to infinity, under standard convexity conditions on the divergence.
    Second, we study the practical implementation of this constrained gradient flow. We propose two time-discrete computational schemes---one explicit, one implicit---, and we prove the convergence of the latter to the OT map as time goes to infinity.
    We then parameterize the OT maps with convexity-constrained neural networks and train them with these discretizations of the constrained gradient flow. We show that this is equivalent to performing a natural gradient descent of the lift of the chosen divergence in the neural networks' parameter space, similarly to drifting generative models.
    Empirically, our scheme outperforms the standard Euclidean gradient descent methods used to train convexity-constrained neural networks in terms of approximation results for the OT map and convergence stability, and it still yields better results than the same approach combined with the widely used \textsc{adam} optimizer.
    
\vspace{2mm}
    
\vspace{2mm}
\end{abstract}

\maketitle
\setcounter{tocdepth}{2}
{
    \hypersetup{linkcolor=black}
    \tableofcontents
}

\newpage

\section{Motivation and introduction}
\label{sec:intro}
\subsubsection*{Motivation} The usual paradigm in machine learning when using a neural network consists in optimizing some loss function directly on the parameter space. This approach is hindered by the non-convexity of the optimization landscape provided by the neural network parametrization, although appropriate architecture choices, such as residual neural networks~\cite{he2016deep,barboni2025understanding}, can partially mitigate these difficulties. However, for more involved settings such as generative adversarial networks, the optimization becomes even more challenging~\cite{goodfellow2020generative}. In such situations, one may design a time-continuous variational problem with global convergence guarantees and use it to \emph{guide} the optimization process, by iteratively training the neural networks to reproduce steps of a gradient descent scheme that discretizes the time-continuous problem. In this article, we explore this principle---which draws from natural gradient schemes~\cite{amari1998natural}---to learn optimal transport maps in the class of convexity-constrained neural networks. We introduce a globally converging gradient flow on the space of gradients of convex functions, and propose an efficient \emph{guiding} (or \emph{drifting}) scheme to obtain approximate solutions to it.

\subsubsection*{A constrained gradient flow} Let $\rho_0$ and $\gamma$ be two probability measures with finite second-order moment. {Optimal transport (OT) maps} between $\rho_0$ and $\gamma$, should they exist, are defined as minimizers of $T\mapsto\int_{\bR^d}\|x-T(x)\|^2\dd\rho_0(x)$
among all elements $T$ of $\LrhozRd$ such that $T_*\rho_0=\gamma$ (see \cref{sec:OT} for details). 
If $\rho_0$ is absolutely continuous, the celebrated \nameref*{thm:brenier}~\cite{brenier1987decomposition} guarantees that there exists a unique OT map between $\rho_0$ and $\gamma$, and that it belongs to the set of gradients of convex functions
\begin{equation}
    \label{eq:K-intro}
    \K\coloneqq\{\nabla\phi\mid\phi\in \dHrhozRd\text{ is convex}\}\subset \LrhozRd.
\end{equation}
Finding the optimal transport map between $\rho_0$ and $\gamma$ therefore amounts to finding $T \in \K$ such that $T_*\rho_0=\gamma$. As a proxy for evaluating the discrepancy between $T_*\rho_0$ and $\gamma$, one could use some divergence $D:\PRd^2\to\bR$, and the problem then boils down to finding
\begin{equation}
    \label{eq:min-D-intro}
    \optmap{\rho_\smallzero}{\gamma}\in\argmin_{\smash{T\in\K}}D(T_*\rho_0\midd\gamma),
\end{equation}
and we write $F:T\mapsto D(T_*\rho_0\midd\gamma)$ this functional to minimize.
Akin to the standard setup of minimizing a functional on a subset of some ambient Hilbert space, it therefore seems reasonable to consider the \emph{gradient flow} of $F$ in $\LrhozRd$ \emph{constrained} to the set $\K$ of optimal transport maps, and hope that suitable convexity conditions on $D$ guarantee its convergence to $\optmap{\rho_\smallzero}{\gamma}$.
Formally, this flow reads 
\begin{equation}
    \label{eq:constrained-flow-intro}
    \partial_tT_t=\proj_{\Tan_{T_t}\!\!\K}(-\nabla F(T_t))
\end{equation}
with $T_0=\id$ and where $\smash{\proj_{\Tan_{T_t}\!\!\K}}$ is the projection onto the (convex) tangent cone of $\K$ at $T_t\in\K$.
We refer to \cref{eq:constrained-flow-intro} as a \emph{constrained gradient flow}.
This flow, should it converge toward $\optmap{\rho_\smallzero}{\gamma}$ with a rate that does not depend on the ambient dimension $d$, would yield a method for estimating OT maps, usable in high-dimensional settings.
In this work, we focus on providing theoretical guarantees on this approach, as well as a computational proof-of-concept of its soundness, using neural networks to parameterize the set $\K$ of gradients of convex functions, with an emphasis on the particular case of the \emph{relative entropy} with respect to some \emph{log-concave} measure.

\subsection{Contributions and outline}

Although the estimation of optimal transport maps is well-explored in low dimensions, the curse of dimensionality appears in higher dimensions, which can be circumvented by paying a higher computational cost.
As far as we are aware, existing approaches reconcile neither statistical guarantees nor computational tractability in high dimensions.
Motivated by the use of neural networks to generate OT maps, we study their estimation using infinite-time limits of \emph{constrained} gradient flows \cref{eq:constrained-flow-intro} in the space of \emph{optimal} transport maps.

On the theoretical side, our main contributions are \cref{thm:flow-defini} and \cref{thm:convergence}, which can be summarized as follows in the particular case of the relative entropy as a functional of choice (answering a question from \citeauthor{modin2016geometry}~\cite[Section 4.1.1]{modin2016geometry}).
\begin{theorem*}[\cref{thm:flow-defini,thm:convergence}, particular case of the relative entropy -- Existence of solutions and convergence for the constrained gradient flow]
    Let $\rho_0\in\PRdac$ be some absolutely continuous probability measure, let $\gamma\in\PRd$ be some strongly-log-concave probability measure, and let $D:\PRd\to\bR$ be the relative entropy with respect to $\gamma$. Then the constrained gradient flow \cref{eq:constrained-flow-intro} admits a solution of time-regularity $H^1$ and it converges exponentially fast to the OT map between $\rho_0$ and $\gamma$, with convergence rate independent of the ambient dimension.
\end{theorem*}
\noindent
We stress that our constrained gradient flows are \emph{not} simply lifts of the standard Wasserstein gradient flows to the space of transport maps, and that their exponential convergence toward the actual OT map between $\rho_0$ and $\gamma$ is therefore new and non-trivial.

Motivated by the theoretical convergence result, we study two time-discrete numerical schemes (one explicit, one implicit) that discretize the (time-continuous) constrained gradient flow \cref{eq:guided-flow-intro}.
The implicit scheme is shown to converge to the OT map (\cref{prop:conv-tau-fixed}), and recovers the continuous scheme \cref{eq:constrained-flow-intro} as the time step goes to zero (\cref{prop:conv-to-continuous}).
We provide a numerical proof-of-concept of the efficiency of those two schemes by parameterizing the set of gradients of convex functions with some convexity-constrained neural network $\theta\mapsto T_\theta$, and we observe that they allow one to reach near-optimal parameters (i.e., to be very close to finding the actual OT maps) significantly more often than the standard convexity-constrained descent schemes.
Although this was expected for the implicit scheme given its good convergence properties, our numerical findings also apply to the explicit one, even in the case of a non-smooth functional such as the entropy; this suggests a possible implicit regularization introduced by the neural networks.

Finally, we note that the constrained gradient flow \cref{eq:constrained-flow-intro} written over a parameterization $\theta\mapsto T_\theta$ of the set $\K$ of OT maps reads
\begin{equation}
    \label{eq:guided-flow-intro}
    \partial_t\theta_t\in\argmin_{\delta\theta\in \Tan_{\theta_t}\!\!\Theta} \int_{\bR^d}\big\|-\nabla F(T_{\theta_t})-\nabla_\theta T_{\theta_t}.\delta\theta\big\|^2 \dd\rho_0,
\end{equation}
where $F:T\mapsto D(T_*\rho_0)$.
We prove the following result, that relates this flow to the family of \emph{natural gradient flows}, known to have good re-parameterization invariance properties. 

\begin{proposition*}[\cref{cor:natural-gf} -- The parameterized constrained gradient flow is a natural gradient flow]
    Let $\Theta\subset\bR^m$ be some parameter space and let $\Theta\ni\theta\mapsto T_\theta$ be a parameterization of a subset of $\K$, differentiable and of injective differential. 
    Let $D:\PRd\to\bR$ be some differentiable functional. Then the parameterized constrained gradient flow \cref{eq:guided-flow-intro} on $\Theta$ is the {natural gradient flow} of $F:T\mapsto D(T_*\rho_0)$ with respect to the $\Lrhoz\!$-metric and the mapping $\theta\mapsto T_\theta$.
\end{proposition*}
\noindent
This last result sheds light on the good computational behavior that our schemes exhibit compared to the standard convexity-constrained descent approaches. As an aside, it also holds for any parameterization of (a subset of) the whole set $\LrhozRd$ of transport maps (\cref{rem:unconstrained}), hinting at the link between drifting models and natural gradient descent schemes.

In closing, we stress that this work does not aim at pushing the numerical state of the art of the estimation of OT maps, but rather at proposing a new method with strong theoretical convergence guarantees.
In that respect, a statistical study of our method would be of great interest; this is left for future work.

\addtocontents{toc}{\protect\setcounter{tocdepth}{0}}
\subsection*{Outline}
This work is organized as follows.
The rest of \cref{sec:intro} is dedicated to providing some background on the literature on learning OT maps (\cref{sec:relatedworks}) and on the technical tools used in this work (\cref{sec:background}).
\addtocontents{toc}{\protect\setcounter{tocdepth}{2}}

\noindent
\cref{sec:theory} provides a theoretical study of the constrained gradient flow.
In \cref{sec:constrained-gf}, we define the flow and establish a few useful results on the structure of the set of OT maps.
In \cref{sec:existence}, we establish the existence of long-time solutions for this constrained gradient flow, while in \cref{sec:convergence} we prove its global convergence toward the OT map under standard convexity assumptions on the functional $D$; assumptions which cover the central case of the relative entropy with respect to some log-concave measure. 

\noindent
\cref{sec:guided} studies the practical implementation of the (time-continuous) constrained gradient flow, via two time-discrete numerical schemes (one explicit, one implicit).
In \cref{sec:unparam}, we show that under standard convexity assumptions on $D$, the implicit scheme converges to the OT map as time goes to infinity, and that it also converges to the time-continuous constrained gradient flow as the time step goes to zero, given a fixed time horizon.
\cref{sec:procedure} formulates the explicit and implicit schemes under the parameterization of $\K$ by neural networks that implement the convexity constraint of the transport maps.
Those schemes are then shown to be discretizations of a \emph{natural gradient flow} in the space of parameters in \cref{sec:natural-gd}. Finally, \cref{sec:numerical} provides a numerical proof-of-concept of the efficiency of our methods to learn OT maps using convexity-constrained neural networks.

\subsection{Related works}
\label{sec:relatedworks}

\subsubsection*{Estimating OT maps in high dimension}
In low dimensions, standard methods for estimating the OT map, such as semi-discrete \cite{kitagawa2019convergence} or discretization of the Monge--Ampère operator \cite{benamou2016monotone,bonnet2022monotone} lead to fast and accurate solutions.
Yet, the estimation of OT maps faces the curse of dimensionality and these methods are not practical in higher dimensions, for which there is still room for improvement.
A standard method to circumvent this consists in reducing the search space \cite{hutter2021minimax} and solving a variational formulation, which we detail now.

\noindent
Finding the OT map amounts to finding a map $T$ that satisfies two conditions.
First, $(i)$ \emph{$T$ must push $\rho_0$ onto $\gamma$}. Implementing this as a hard constraint is difficult in practice \cite{korotin2021neural,uscidda2023monge} and this paper fits in a line of works that focuses on relaxing it using a penalization term of the form $T\mapsto D(T_*\rho_0\midd\gamma)$, where $D$ is some divergence on $\PRd$ \cite{lu2020large,xie2019scalable,bousquet2017optimal,balaji2020robust,uscidda2023monge}, which greatly facilitates the optimization procedure.
Second, $(ii)$ \emph{$T$ must be optimal}, that is, of minimal transport cost. This condition has been used as a soft constraint, using either the primal \cite{leygonie2019adversarial,liu2021learning,lu2020large}, semi-dual \cite{divol2025optimal,vacher2022parameter,muzellec2024near} or dual \cite{makkuva2020optimal,seguy2017large} formulation of the OT problem, allowing for some degree of sub-optimality of the learned transport map. Yet, in some cases, one might want to enforce the optimality exactly \cite{ait2003nonparametric,varian1982nonparametric,kuosmanen2008representation,chetverikov2018econometrics}, and this is the point of view we adopt in this work.
In practice, one may parameterize the set of OT maps and try to find $T$ minimizing the penalization term mentioned above.
Linear parameterizations introduce computational difficulties in high dimensions \cite{mirebeau2014adaptiveanisotropichierarchicalcones}.
This can be mitigated by using (convexity-constrained) neural networks:
Input Convex Neural Networks (ICNNs) have attracted a lot of attention in the recent years \cite{amos2017input,richter2021input,gagneux2025convexity,bunne2022supervised}, while other expressive parameterizations such as Log-Sum-Exp (LSE) networks \cite{calafiore2019log} or Max-Affine models \cite{ghosh2021max} seem to be less used in practice.
Although neural networks can be shown to be expressive enough \cite{barron1994approximation}, using them comes at the price of non-linearity, which implies that standard gradient flows may converge to spurious local minima.
The method we present in this work seamlessly adapts to any of these parameterization choices.
See also \cite{hurault2023convergent,chaudhari2023learning,saremi2019approximating} for learning gradients of convex functions, and \cite{korotin2022neural,korotin2021neural,amos2022amortizing,vesseron2024neural,fan2021scalable,drygala2025learning,chaudhari2025gradnetot} to do so in the specific context of finding OT maps.

\subsubsection*{Flows and curves for finding OT maps}
Our method consists in lifting some divergence on $\PRd$ (e.g., the relative entropy) to the space of transport maps and performing its constrained gradient flow to the subset of \emph{optimal} transport maps. This has been suggested by \citeauthor{modin2016geometry} in \cite[Section 2.2.3]{modin2016geometry}, with details and numerical experiments in the finite-dimensional particular case of Gaussian measures. Our point of view is to use the cone structure of the set of OT maps to define our flow, benefiting from well-known results for the convergence of gradient flows of convex functionals on Hilbert spaces \cite{de1993new,rossi2006gradient}. This method also relates to \cite{jiang2025algorithms}, where flows are performed on a subset of the set of OT maps satisfying the very restrictive condition of \emph{compatibility}~\cite{boissard2015distribution},
or, in a finite-dimensional setting, to the literature on gradient flows constrained to submanifolds of Euclidean spaces \cite{hauswirth2016projected,alvarez2004hessian}.
See also \cite{angenent2003minimizing,modin2016geometry,morel2022turning,vesseron2024neural} for methods aiming to improve the optimality of a learned sub-optimal transport map.
One may also mention the method of continuity \cite{de2014monge,gonzalez2024linearization}, where an OT map is obtained by solving the linearization of the Monge--Ampère equation.

\subsubsection*{Link with drifting generative models}
Independently of our work, drifting models \cite{deng2026generative,cao2026gradient} have recently been introduced for generative modeling. These methods consist in performing the gradient flow of a modified Maximum Mean Discrepancy (MMD) via a natural gradient descent scheme on the space of maps, in a way that is closely related to \cref{eq:guided-flow-intro}, but without the convexity constraint inherent to our approach, as we specifically seek for OT maps.
Their proposed optimization method corresponds to the explicit scheme we introduce in \cref{sec:guided}. Our numerical method differs by the use of convexity-constrained neural networks, which enforces the optimality constraint.

\addtocontents{toc}{\protect\setcounter{tocdepth}{0}}
\subsection*{Notation}
In this work, we adopt the following notation.
\addtocontents{toc}{\protect\setcounter{tocdepth}{2}}

\noindent \emph{Optimal transport.}
Let $d\geq1$.
\begin{itemize}[leftmargin=7mm]
    \item $\PRd$ is the set of probability measures on $\bR^d$ {with finite second-order moment}.
    \item If some $\rho\in\PRd$ is {absolutely continuous} with respect to some $\gamma\in\PRd$, we denote by $\frac{\dd \rho}{\dd \gamma}$ the corresponding Radon--Nikodym derivative.
    \item $\PRdac$ is the subset of $\PRd$ of {absolutely continuous measures} with respect to the $d$-dimensional Lebesgue measure, which we write $\dd x$.
    \item The {pushforward} $T_*\rho$ of some $\rho\in\PRd$ by some measurable map $T:\bR^d\to\bR^d$ is the probability measure defined on Borel sets $A$ by $T_*\rho(A)\coloneqq \rho(T^{-1}(A))$.
    \item $\LrhoRd$ is the Hilbert space of measurable functions $T:\bR^d\to\bR^d$ that are squared-integrable with respect to some $\rho\in\PRd$, endowed with its norm $\smash{\|\cdot\|_{\Lrho}}$ and scalar product $\smash{\langle\cdot,\cdot\rangle_{\Lrho}}$; $\dHrhoRd$ is the space of functions whose distributional derivative is in $\LrhoRd$.
    \item The optimal transport map between some $\rho$ and $\gamma$ in $\PRd$ is written $\optmap{\rho}{\gamma}\in \LrhoRd$.
    \item $\div$ denotes the divergence (see \cref{sec:divergence} for a definition).
\end{itemize}

\noindent \emph{Riemannian geometry.}
Let $M$ be a Riemannian manifold with Riemannian metric $g$ ($M$ can be infinite-dimensional, with a strong metric \cite{schmeding2022introduction}).
\begin{itemize}[leftmargin=7mm]
    \item The metric $g$ induces on the tangent space $T_pM$ at some $p\in M$ a scalar product, hence a norm, that we write $\langle\cdot,\cdot\rangle_g$ and $\|\cdot\|_g$.
    \item The differential of some functional $\ell:M\to\bR$ at some $p\in M$ is written $d_p\ell\in T_p^*M$,
and its Riemannian gradient $\grad^g_M\ell(p)$ is defined as the unique element of $T_pM$ such that $g_p(\grad^g_M \ell(p),\cdot)=d_p\ell[\cdot]$.
\end{itemize}

\subsection{Technical background}
\label{sec:background}
In this section, we review various preliminaries in optimal transport (\cref{sec:OT}), as well as in convex analysis in Hilbert spaces (\cref{sec:diff-hilbert}) and in the Wasserstein space $\PRd$ (\cref{sec:diff-wass}). Some additional notions can also be found in \cref{appendix:omitted_notions}.

\subsubsection{Optimal transport and optimal transport maps}
\label{sec:OT}
Let $\rho_0,\gamma\in \PRd$ be two probability measures on $\bR^d$ (with finite second-order moment).
The (Monge) \emph{optimal transport (OT) cost} between $\rho_0$ and $\gamma$ is defined as \cite{monge1781memoire,villani2009optimal,santambrogio2015optimal,peyre2019computational}
\begin{equation}
    \label{eq:monge}
    \tag{\textsc{ot}}
    \operatorname{OT}_{\text{Monge}}(\rho_0,\gamma)^2=\inf_{T\in\cT(\rho_0,\gamma)}\int_{\bR^d}\|x-T(x)\|^2\dd\rho_0(x), 
\end{equation}
where $\cT(\rho_0,\gamma)$ is the set of \emph{transport maps}, that is, measurable maps $T\in\LrhozRd$ such that the pushforward measure $T_*\rho_0$ is equal to $\gamma$. 
A solution $\optmap{\rho_\smallzero}{\gamma}$ to \cref{eq:monge} is called an \emph{optimal transport map}, or \emph{Monge map}. 
However, \cref{eq:monge} might not admit a solution, and the set $\cT(\rho_0,\gamma)$ may be empty. 
One may relax the Monge problem \cref{eq:monge} to that of \citeauthor{kantorovich1942translocation}~\cite{kantorovich1942translocation}, which serves as the usual definition of the well-known \emph{Wasserstein distance}:
\begin{equation}
    \label{eq:kantorovitch}
    \tag{$\W_2$}
    \W_2(\rho_0,\gamma)^2=\min_{\pi\in\Pi(\rho_0,\gamma)}\iint_{\bR^d\times\bR^d}\|x-y\|^2\dd\pi(x,y), 
\end{equation}
where the minimization is done over the set $\Pi(\rho_0,\gamma)$ of probability measures $\pi\in\PRdRd$ admitting $\rho_0$ and $\gamma$ as marginals.
Such elements $\pi$ are called \emph{transport plans}, and a solution to \cref{eq:kantorovitch} is called an \emph{optimal transport plan}, their set being denoted by $\Pi_\text{o}(\rho,\gamma)$. 
Transport maps are a special case of transport plans, namely, plans of the form $(\id,T)_*\rho_0$. 
The Wasserstein distance makes $\PRd$ a metric space, and metrizes \emph{weak convergence} in $\PRd$ (denoted by $\rho_n\rightharpoonup\rho$ in this work), that is, \emph{narrow convergence} (convergence against bounded continuous test functions) together with convergence of the second-order moments \cite[Theorem~6.9]{villani2009optimal}.

Under the assumption that $\rho_0$ has a density with respect to the Lebesgue measure, one can ensure the existence and uniqueness of an optimal transport plan, and guarantee that it is actually induced by a map, as shown by the following celebrated theorem of \citeauthor{brenier1987decomposition}~\cite{brenier1987decomposition}:
\begin{theorem*}[Brenier's theorem] 
    \label{thm:brenier}
    Let $\rho_0\in\PRdac$ be an absolutely continuous probability measure. Then for any $\gamma\in\PRd$ there exists a solution to the \cref{eq:kantorovitch} problem, it is unique (up to a set of $\rho_0$-measure zero), and it is induced by a map $\optmap{\rho_\smallzero}{\gamma}:\bR^d\to\bR^d$ which is the unique (up to a set of $\rho_0$-measure zero) gradient of a convex function $\phi:\bR^d\to\bR$ pushing $\rho_0$ onto $\gamma$.
\end{theorem*}
\noindent
As a direct consequence, if $\rho_0\in\PRdac$, the gradient $\nabla\phi\in\LrhozRd$ of any convex function $\phi:\bR^d\to\bR$ is the optimal transport map between $\rho_0$ and the pushforward measure $(\nabla\phi)_*\rho_0$.
In this work, we consider the case where $\rho_0$ is absolutely continuous; the search for an optimal transport between $\rho_0$ and some $\gamma\in\PRd$ therefore reduces to searching for an optimal transport \emph{map}.
A case of interest will also be that of a \emph{$(\lambda$-)log-concave} target measure $\gamma$, that is, a measure $\gamma$ with a Radon--Nikodym derivative with respect to the Lebesgue measure that writes $\smash{\frac{\dd\gamma}{\dd x}}=e^{-V}$, with $V:\bR^d\to\bR$ a ($\lambda$-)convex function. Every log-concave measure finite moments of all orders \cite[Appendix~B.1]{bobkov2019one}; in particular, every log-concave measure belongs to $\PRd$.

\noindent 
Let us conclude this section by noting that for any $T,S\in\LrhozRd$, the transport plan $(T,S)_*\rho_0$ has $T_*\rho_0$ and $S_*\rho_0$ as marginals; hence its sub-optimality for the optimization problem \cref{eq:kantorovitch} directly gives $\W_2(T_*\rho_0,S_*\rho_0)^2\leq\|T-S\|_{\Lrhoz}^2$.

\subsubsection{Differential calculus, gradient flows, and convexity in Hilbert spaces}
\label{sec:diff-hilbert}
Let $\cH$ be a Hilbert space and $F:\cH\to\bR$ some functional. 
The \emph{Fréchet subdifferential} $\partial ^-F(x)$ of $F$ at some $x\in\cH$ is defined as the set of $\xi\in \cH$ such that
\begin{equation}
    F(y)-F(x)\geq \langle \xi,y-x\rangle +o(\|y-x\|)\qquad\text{as $y\to x$.}
\end{equation}
Furthermore, we write $\partial^\circ F(x)$ the unique element of minimal norm of $\partial^- F(x)$.
The \emph{Fréchet superdifferential} of $F$ at $x$ is defined as $\partial^+F(x)\coloneqq-\partial^-(-F)(x)$. The functional $F$ is said to be \emph{differentiable} at $x\in\cH$ if $\partial^- F(x)\cap\partial^+F(x)$ is non-empty. 
In this case, the element of minimal norm in this set is called the \emph{Fréchet gradient} of $F$ at $x$ and written $\nabla F(x)$, and one has
\begin{equation}
    F(y)-F(x)= \langle \nabla F(x),y-x\rangle +o(\|y-x\|)\qquad\text{as $y\to x$.}
\end{equation}
A \emph{gradient flow} of a differentiable functional $F$ is a curve $u \in H^1([0,t_{\text{max}}],\cH)$ for some $t_{\text{max}} > 0$ such that $u_0 \in \cH$ and
\begin{equation}
    \dot u_t =- \nabla F(u_t) \qquad \text{ for a.e.~} t \in (0,t_{\text{max}}).
\end{equation}
If $K\subset \cH$ is some convex subset of the ambient Hilbert space $\cH$, then the \emph{gradient flow of $F$ constrained to $K$} is a curve $u\in H^1([0,t_{\text{max}}],\cH)$ for some $t_{\text{max}} > 0$ such that $u_0 \in K$ and
\begin{equation}
    \dot u_t =-\proj_{\Tan_{u_t}\!\!K} (\nabla F(u_t) ) \qquad \text{ for a.e.~} t \in (0,t_{\text{max}}),
\end{equation}
where $\Tan_{u_t}K$ is the tangent cone of $K$ at $u_t\in K$ and $\proj$ the usual projection onto convex sets (see \cref{sec:cones} for a definition of both).

\begin{remark}[Gradient flow and inner product]
\label{rem:gf_geometry}
Those gradient flows depend on the gradient on the Hilbert space $\cH$, hence on the choice of an inner product on $\cH$.
In \cref{sec:natural-gd}, we work with a gradient flow on some $\Theta\subset\bR^m$ for an inner product that differs from the Euclidean one. 
\end{remark}
\noindent
While the existence, uniqueness, and asymptotic behavior of gradient flows in not trivial in general \cite{rossi2006gradient}, things become much easier if $F$ exhibits some \emph{convexity} properties \cite[Section\! 1.4]{ambrosio2008gradient}.
Namely, $F$ is said to be \emph{$\lambda$-convex} for $\lambda \in \bR$ on a convex subset $K\subset\cH$ if for all $x,y\in K$ and all $t\in[0,1]$,
\begin{equation}
    \label{eq:cvx}
    F((1-t)x+ty)\leq (1-t)F(x)+tF(y)-\frac\lambda2t(1-t)\|x-y\|^2.
\end{equation}
Eventually $F$ is said to be $\lambda$-\emph{star-convex} on $K$ \emph{around $x^\star \in \cH$} if the previous inequality is true for all $x\in K$ and $y=x^\star$.

\subsubsection{Differential calculus, gradient flows, and convexity in \texorpdfstring{$\PRd$}{P2(Rd)}}
\label{sec:diff-wass}
Since $\PRd$ does not enjoy a linear structure, the notions of the previous subsection do not apply faithfully; yet, they can be adapted and will play an important role in this work.
See for instance \cite[Definitions~2.11 and~2.12]{bonnet2019pontryagin}, \cite[Chapter~10]{ambrosio2008gradient} or \cite[Definition~2.1]{chow2019partial}.
The \emph{Wasserstein subdifferential} $\partial^- D(\rho)$ of a functional $D : \PRd \to \bR$ at $\rho\in\PRd$ is defined as the set of $\xi\in\LrhoRd$ such that
\begin{equation}
    D(\gamma)-D(\rho)\geq \inf_{\pi\in\Pi_\text{o}(\rho,\gamma)}\iint_{\bR^d\times\bR^d}\langle \xi(x),y-x\rangle\dd\pi(x,y)+o(\W_2(\rho,\gamma))\qquad\text{as $\gamma\rightharpoonup\rho$.}
\end{equation}
Furthermore, we write $\partial^\circ D(\rho)$ the (unique) element of minimal norm of $\partial^- D(\rho)$.
The \emph{Wasserstein superdifferential} of $D$ at $\rho$ is defined as $\partial^+D(\rho)=-\partial^- (-D)(\rho)$. 
The functional $D$ is then said to be \emph{Wasserstein differentiable} at $\rho\in\PRd$ if $\partial^- D(\rho)\cap\partial^+D(\rho)$ is non-empty. In this case, the element of minimal norm in this set is called the 
\emph{Wasserstein gradient} of $D$ at $\rho$ and written $\nablaa D(\rho) \in\LrhoRd$, and one has
\begin{equation}
    \label{eq:gradient-W}
    D(\gamma)-D(\rho)= \iint_{\bR^d\times\bR^d}\langle \nablaa D(\rho)(x),y-x\rangle\dd\pi_\gamma(x,y)+o(\W_2(\rho,\gamma))\qquad\text{as $\gamma\rightharpoonup\rho$,}
\end{equation}
for all selections $(\pi_\gamma)_\gamma$ of the family of sets $(\Pi_\text{o}(\rho,\gamma))_\gamma$.
When it exists, the Wasserstein gradient belongs to the \emph{Wasserstein tangent space} \cite[Definition~8.4.1]{ambrosio2008gradient}
\begin{equation}\label{eq:tangent_space_PRd}
    \operatorname{Tan}_\rho\PRd=\overline{\{\nabla\phi\mid\phi\in C^\infty_c(\bR^d)\}}^{\Lrho},
\end{equation}
see for instance \cite[Theorem~3.10, Definition~3.11]{gangbo2019differentiability} and \cite[Proposition~8.5.4]{ambrosio2008gradient}.
Additionally, under some assumptions on $D$\footnote{For instance, if $D$ has a first variation $\smash{\frac{\delta D}{\delta\rho}}$ that is differentiable and if $D$ is a \emph{regular} functional in the sense of \cite[Definition~10.1.4]{ambrosio2008gradient}. See \cref{prop:gradient-first-var} for a short proof.}, then for all $\rho\in\PRd$,
\begin{equation}
    \label{eq:wass-gradient}
    \nablaa D(\rho)=\nabla \frac{\delta D}{\delta\rho}(\rho),
\end{equation}
where $\smash{\frac{\delta D}{\delta\rho}}$ is the first variation of $D$.
An absolutely continuous curve $(\rho_t)_t$ in $\PRd$ is a \emph{Wasserstein gradient flow} of $D$ if it satisfies the continuity equation
\begin{equation}
    \label{eq:CE}
    \partial_t\rho_t =- \div(\rho_t v_t) \qquad\text{with }v_t = -\nablaa D(\rho_t)
\end{equation} 
in a weak sense for a.e.~$t > 0$ \cite[Equation~(8.3.8)]{ambrosio2008gradient}. 
As for gradient flows in Hilbert spaces, Wasserstein gradient flows are easier to study whenever the functional $D$ exhibits convexity properties, this time along a specific type of curves, namely, geodesics and generalized geodesics. 
For concision's sake, we introduce these notions only for absolutely continuous measures and refer to \cref{appendix:generalizedGeod} for a presentation of the general setting, following the framework of \cite[Chapter~9]{ambrosio2008gradient}. 

\begin{definition}[Convexity along (generalized) geodesics with absolutely continuous measures] \label{def:geod_cvx}
    Let $\rho_0 \in \PRdac$ and $\rho_1,\rho_2 \in \PRd$. Let $\optmap{\rho_\smallzero}{\rho_\smallone}$ be the OT map between $\rho_0$ and $\rho_1$ according to \nameref*{thm:brenier}.
    The \emph{geodesic} between $\rho_0$ and $\rho_1$ is the curve $(\rho_t)_{t \in [0,1]}$ given by
    \begin{equation}
    \label{eq:geodesic-ac}
    \rho_t=[(1-t)\id+t\optmap{\rho_\smallzero}{\rho_\smallone}]_*\rho_0,
\end{equation}
in the sense that $\W(\rho_t, \rho_t) = |t-s| \W(\rho_0,\rho_1)$ for all $s,t \in [0,1]$.
The \emph{generalized geodesic} between $\rho_1$ and $\rho_2$ \emph{with anchor point} $\rho_0$ is the curve $(\rho_t)_{t \in [0,1]}$ given by 
\begin{equation}
    \label{eq:gen-geodesic-ac}
    \rho_t=[(1-t)\optmap{\rho_\smallzero}{\rho_\smallone}+t\optmap{\rho_\smallzero}{\rho_\smalltwo}]_*\rho_0.
\end{equation}
\noindent A functional $D : \PRd \to \bR$ is said to be \emph{$\lambda$-convex along generalized geodesics} for $\lambda\in\bR$ if 
\begin{equation}
    \label{eq:cvx-wass-gen-ac}
    D(\rho_t)\leq (1-t)D(\rho_1)+tD(\rho_2)-\frac\lambda2t(1-t)\|\optmap{\rho_\smallzero}{\rho_\smallone}-\optmap{\rho_\smallzero}{\rho_\smalltwo}\|_{\Lrhoz}^2,
\end{equation}
for all curves of the form \cref{eq:gen-geodesic-ac}.
If the above formula holds only when $\rho_0=\rho_1$ (in which case $\optmap{\rho_\smallzero}{\rho_\smallone} = \id$ and \cref{eq:gen-geodesic-ac} is a geodesic, of the form \cref{eq:geodesic-ac}), $D$ is said to be \emph{$\lambda$-convex along geodesics}. 
\end{definition}

Whenever the functional $D$ is $\lambda$-convex along geodesics in $\PRd$, existence and uniqueness of gradient flows are guaranteed by \cite[Theorem~11.1.4]{ambrosio2008gradient}; if furthermore $\lambda > 0$, then $\rho_t$ converges exponentially fast toward the (then unique) minimizer of $D$. 
Convexity along \emph{generalized} geodesics is a stronger condition, and enables sharper results on gradient flows and their discretizations \cite[Chapter~11]{ambrosio2008gradient}. 
All functionals we consider in this work are convex along \emph{generalized} geodesics in the general sense of \cite[Chapter~9]{ambrosio2008gradient} (see \cref{appendix:generalizedGeod}), hence in the weaker sense of \cref{def:geod_cvx}, as the latter only asks for convexity when the source or anchor measures are absolutely continuous.

\begin{example}[Relative entropy]
\label{ex:entropy}
A celebrated example of a functional that motivated the study of gradient flows in $\PRd$ is the \emph{relative entropy}, also known as \emph{Kullback--Leibler divergence} \cite{kullback1951information}. 
The relative entropy of $\rho$ with respect to $\gamma \in \PRd$ is defined as
\begin{equation}
    \label{eq:relative-entropy}
    D(\rho)\coloneqq H(\rho\midd\gamma)= \int_{\bR^d}\log\Big(\frac{\dd\rho}{\dd\gamma}\Big)\dd\rho
\end{equation}
if $\rho$ has a density with respect to $\gamma$, else $H(\rho\midd\gamma)=\infty$. 
The relative entropy with respect to $\gamma$ is $\lambda$-convex along generalized geodesics if and only if $\gamma$ is $\lambda$-log-concave \cite[Theorem~9.4.11]{ambrosio2008gradient}.
Writing $\gamma \propto e^{-V}$ for some potential $V : \bR^d \to \bR \cup \{\infty\}$, 
a Wasserstein gradient flow $(\rho_t)_t$ of \cref{eq:relative-entropy} satisfies the following continuity equation, also known as the \emph{Fokker--Planck equation}, 
\begin{equation}
    \label{eq:entropy-flow}
    \partial_t\rho_t=\div(\rho_t\nablaa H(\rho_t\midd\gamma)),\qquad \text{where }\nablaa H(\rho\midd \gamma)=\nabla\log\rho+\nabla V.
\end{equation}
Setting for instance $V = 0$ retrieves the {heat equation}. 
Another common choice is $V(x) = \smash{\frac{\|x\|^2}2}$, which amounts to $\gamma= N(0_d,I_d)$.
Under some conditions on the reference measure $\gamma$ (its log-concavity, or more generally the logarithmic Sobolev inequality \cite{stam1959some,gross1975logarithmic}), the Wasserstein gradient flow \cref{eq:entropy-flow} of $H$ has been shown to converge exponentially fast toward $\gamma$, for instance in terms of the Wasserstein distance \cref{eq:kantorovitch} \cite{bakry2006diffusions,bakry2014analysis,ambrosio2008gradient}.
\end{example}

\begin{example}[MMD] \label{ex:MMD}
The \emph{Maximum Mean Discrepancy (MMD)} between two probability measures $\rho$ and $\gamma$  in $\PRd$ is defined as
\begin{equation}
    \label{eq:MMD}
    D(\rho)\coloneqq \operatorname{MMD}_k(\rho,\gamma)=\frac12\iint_{\bR^d\times\bR^d} k(x,y)\dd (\rho - \gamma)(x)\dd( \rho- \gamma)(y),
\end{equation}
where $k : \bR^d \times \bR^d \to \bR$ is a symmetric positive-definite (or conditionally positive-definite) kernel. Contrary to the relative entropy \cref{eq:relative-entropy}, the MMD is finite (under moments assumptions) even when $\rho$ and $\gamma$ have disjoint supports or when the measures are atomic.
A standard choice is $k(x,y) = -\|x-y\|$, in which case $\operatorname{MMD}_k$ is referred to as the \emph{energy distance MMD} and has very good behavior regarding the convergence of its gradient flow \cite{chizat2026mmd}.
Other choices of kernels are possible, see for instance \cite{glaunes2004diffeomorphic,gretton2006kernel,hagemann2023posterior,boufadene2025global,hertrich2023generative}.
The MMD has the nice property that its \emph{sample complexity} (the rate of convergence of its value between some measure and its empirical counterpart when the number $n$ of samples goes to infinity) is independent of the dimension and scales as $O(1/\sqrt n)$ \cite{gretton2006kernel}. This is in stark contrast to the \cref{eq:kantorovitch} distance, which suffers from the curse of dimensionality and whose sample complexity scales as $O(1/n^{1/d})$ \cite{weed2019sharp}. 
\end{example}

\section{The constrained gradient flow} \label{sec:theory}

Let $\rho_0\in\PRdac$ and $\gamma\in\PRd$.
In this section, 
we propose a new evolution equation in $\LrhozRd$, which, under standard convexity assumptions, converges as $t\to\infty$ to the OT map $\optmap{\rho_\smallzero}{\gamma}$ between $\rho_0$ and $\gamma$. We refer to this evolution in $\LrhozRd$ as a \emph{constrained gradient flow}. It is defined in \cref{sec:constrained-gf}; \cref{sec:existence} then focuses on proving the existence of its solutions, and \cref{sec:convergence} on proving its convergence to the OT map.

\subsection{Definition of the constrained gradient flow}
\label{sec:constrained-gf}

Let $D:\PRd\to\bR$ be some divergence that assesses whether a given probability measure $\rho \in \PRd$ is close to $\gamma$; for instance, the entropy \cref{eq:relative-entropy} or the MMD \cref{eq:MMD}, both relative to $\gamma$. 
As a proxy for solving the OT problem between $\rho_0$ and $\gamma$, recall that we consider the constrained optimization problem \cref{eq:min-D-intro}: 
one wishes to find some transport map $\optmap{\rho_\smallzero}{\gamma}$ that minimizes $F:T\mapsto D(T_*\rho_0)$ (which guarantees that $\optmap{\rho_\smallzero}{\gamma}{}_*\rho_0=\gamma$) while belonging to the cone of gradients of convex functions
\begin{equation}
    \label{eq:krho}
    \tag{\ref*{eq:K-intro}}
    \K=\{\nabla\phi\mid\phi\in \dHrhozRd\text{ is convex}\}\subset \LrhozRd.
\end{equation}
Hence the problem amounts to minimizing $F$ over the cone $\K$.
Akin to the standard setup of minimizing a functional on a submanifold of some ambient Hilbert space, we consider the \emph{gradient flow} of $F$ in $\LrhozRd$ \emph{constrained to $\K$}, and hope that suitable convexity conditions on $F$ or on $D$ guarantee its convergence to $\optmap{\rho_\smallzero}{\gamma}$. 

\subsubsection{The set \texorpdfstring{$\K$}{K}, its tangent cone, and the lifted functional}
We first describe the set $\K$ and the functional $F$ defined above.
\begin{proposition}[$\K$ is convex and closed]
    \label{prop:convex-closed}
    The set $\K$ is a closed convex subset of the Hilbert space $\LrhozRd$.
\end{proposition}
\begin{proof}
    The convexity of $\K$ directly follows from the convexity of the space of convex functions on $\bR^d$.
    Let us then show its closedness. Let $(T_n)_n\subset\K$ be a sequence of elements of $\K$ that strongly converges to some $T\in\LrhozRd$. Let us write $\rho_n\coloneqq T_n{}_*\rho_0$. By continuity of the pushforward mapping (given by the inequality $\W_2(T_*\rho_0,S_*\rho_0)\leq \|T-S\|_{\Lrhoz}$ for all $T,S\in\LrhozRd$), the sequence $(\rho_n)_n\subset\PRd$ converges weakly to $\gamma\coloneqq T_*\rho_0$. Since $\rho_0$ has a density, by \cite[Corollary~5.23]{villani2009optimal}, $(T_n)_n$ converges in probability to the optimal transport map $T'\in\K$ between $\rho_0$ and $\gamma$. Since $(T_n)_n$ converges strongly to $T$, it also converges in probability to $T$ and the uniqueness of the limit in probability yields $T(x)=T'(x)$ for $\rho_0$-a.e.~$x\in\bR^d$; hence $T=T'$ in $\LrhozRd$ and $T$ therefore belongs to $\K$.
\end{proof}

Since $\K$ is a closed convex in $\LrhozRd$, it admits a (Clarke) tangent cone (see \cref{sec:cones} for a reminder on this matter in arbitrary Hilbert spaces).
\begin{definition}[Tangent cone of $\K$]
    \label{lem:tangentconeexplicitly}
    The \emph{tangent cone} of $\K$ at some $T\in\K$ is defined as
    \begin{equation}
        \Tan_{T}\!\K=\overline{\big\{w\in \LrhozRd\mid \exists t_0>0\text{ s.t.~}\forall t\leq t_0,\, T+tw\in\K \big\}}^{\Lrhoz}.
    \end{equation}
\end{definition}
\noindent 
The definition \cref{eq:krho} of $\K$ allows us to write its tangent cone more explicitly, as follows.
\begin{lemma}[Characterization of the tangent cone of $\K$]
    Let $T\in\K$, and write $T=\nabla\phi$ with $\phi\in\dHrhozRd$ convex.
    The tangent cone of $\K$ at $T$ is equal to
\begin{equation}
    \label{eq:tangentcone}
    \Tan_{T}\!\K=\overline{\big\{\nabla\bfp,\,\bfp\in\dHrhozRd\mid \exists t_0>0\text{ s.t.~}\forall t\leq t_0,\, \phi+t\bfp\text{ convex} \big\}}^{\Lrhoz}.
\end{equation}
\end{lemma}
\begin{remark}[Some intuition on the tangent cone]
\label{rem:tangent-cone}
It is worth expanding a bit on the characterization \cref{eq:tangentcone} of the tangent cone of the closed convex cone $\K$, which differs whenever we examine it at some $T$ which is $(i)$ in the interior\footnote{In this remark, $\K$ is considered here as a subset of the topological space $\overline{\{\nabla\bfp,\,\bfp\in\dHrhozRd\}}^{\Lrhoz}$. Indeed, $\K$ has empty interior in the whole $\LrhozRd$, and its boundary in $\LrhozRd$ is itself.} of $\K$ or $(ii)$ on its boundary.
\begin{enumerate}[$(i)$,leftmargin=*]
    \item \textit{In the interior of $\K$.} Let $T\coloneqq\nabla\phi$ where $\phi$ is some $C^2$ strictly convex function, that is, such that $\nabla^2\phi>0$ on $\bR^d$. Then, for any $\bfp$ of $C^2$-regularity, there exists some small enough $t_0>0$ such that $\phi+t\bfp$ remains convex for all $t<t_0$, and therefore the gradient of any such $\bfp$ belongs to $\Tan_{T}\!\K$.
    \item \textit{On the boundary of $\K$.}
    Let $T\coloneqq\nabla\phi$ where $\phi$ is piecewise affine and let $\bfp\coloneqq-\frac12\|x\|^2$. Then $w\coloneqq\nabla\bfp$ does not belong to the tangent cone at $T$, since adding $t\bfp$ to $\phi$ results in a function $\phi+t\bfp$ that is never convex for any $t>0$.
    This example still holds in the more general setting of functions $\phi$ that are convex, but not strictly convex on the whole $\bR^d$, and strictly concave functions $\bfp$. \qedhere
\end{enumerate}
\end{remark}

Finally, let us introduce some notation and describe the functional $F:T\mapsto D(T_*\rho_0)$ mentioned in the introduction of this section, which will be central in this work.
\begin{definition}[Lifted functional]
\label{def:lifted}
Let $D:\PRd\to\bR$ be some functional on $\PRd$.
The \emph{lifted functional} $F:\LrhozRd\to\bR$ is the functional $F\coloneqq D\circ\pi$, where $\pi:T\mapsto T_*\rho_0$.
\end{definition}
\noindent Observe that this lifted functional $F$ is constant along the fibers of $\pi$, that is, along all the
\begin{equation}
    \pi^{-1}(\rho)=\{T\in\LrhozRd\mid T_*\rho_0=\rho\}
\end{equation}
for $\rho\in\PRd$.
The lifted functional $F$ also inherits many properties from $D$ which will prove useful in this work; those results are stated and proved in \cref{sec:lifting}.
Of particular importance, whenever $D$ is differentiable in $\PRd$, $F$ is differentiable in $\LrhozRd$ as well and its gradient is given by $\nabla F(T)=\nablaa D(T_*\rho_0)\circ T$ (see \cref{lem:frechet-diff}). In order to perform the gradient flow of $F$ \emph{constrained to $\K$}, one needs to project this gradient onto the tangent cone of $\K$.
This motivates the first definition of the next section, which is the standard definition of constrained gradient flows in Hilbert spaces (see \cref{sec:diff-hilbert}).

\subsubsection{The constrained gradient flow: definition and first properties}
\begin{definition}[Constrained gradient flow]
\label{def:constrained-flow}
Let $D:\PRd\to\bR$ be some functional on $\PRd$
and $F=D\circ\pi$ its lifted functional.
A \emph{constrained gradient flow} of $F$ in $\LrhozRd$ is a curve $(T_t)_t$ in $H^1([0,t_{\text{max}}],\LrhozRd)$ that is solution of
\begin{equation}
    \label{eq:constrained-flow}
    \tag{\textsc{c}ons.\textsc{gf}}
    \left\{\begin{aligned}
         &\ T_0=\id   \\
         &\ \partial_t T_t
         = \proj_{\Tan_{T_t}\!\!\K}(-\nablaa D(T_t{}_*\rho_0)\circ T_t) \qquad \text{for a.e.~}t\in (0,t_{\text{max}}),
    \end{aligned}\right.
    \vphantom{\left\{\begin{aligned}
         &\ T_0=\id   \\
         &\  {\int_{\bR^d}}
    \end{aligned}\right.}
\end{equation}
where
$\proj$ is the usual projection onto convex sets.
\end{definition}

\noindent
Note that since for all $T\in\K$, the set $\Tan_{T}\!\K$ is a nonempty closed convex subset of the Hilbert space $\LrhozRd$, the projection onto $\Tan_{T}\!\K$ exists and is unique, making $\partial_tT_t$ well-defined in \cref{eq:constrained-flow}.

\begin{remark}[On the necessity of the projection step]
    \label{rem:flowmap1}
    The projection step in \cref{eq:constrained-flow} is necessary, since updates $-\nablaa D(T_t{}_*\rho_0)\circ T_t$ do not, in general, make $T_t$ stay in $\K$. One can indeed find counter-example measures $\rho_0$ and $\gamma$ when $D$ is the relative entropy \cref{eq:relative-entropy} with respect to $\gamma$ and $d\geq2$ \cite{tanana2021comparison,lavenant2022flow,kim2012generalization}.
\end{remark}

\noindent
Let us now establish some properties satisfied by the increment $\partial_tT_t$ in \cref{eq:constrained-flow}.

\begin{lemma}[First properties of $\partial_t T_t$]
    \label{lem:charac}
    Let $v_t\coloneqq -\nablaa D(T_t{}_*\rho_0)$ and $w_t\coloneqq \partial_tT_t$ be the solution of the constrained gradient flow \cref{eq:constrained-flow}.
    Then
    \begin{enumerate}[$(i)$]
        \item \label{item:charac1} $\langle v_t\circ T_t,w_t\rangle_{\Lrhoz}=\|w_t\|^2_{\Lrhoz}$;
        \item \label{item:charac2} for all $S\in\K$, $\langle v_t\circ T_t-w_t,S-T_t\rangle_{\Lrhoz}\leq0$.
    \end{enumerate}
\end{lemma}
\begin{proof}
    Since $\Tan_{T_t}\!\K$ is a nonempty closed convex set by \cref{prop:convex-closed}, one can use the characterization of the projection on closed convex sets \cite[Theorem~5.2]{brezis2011functional}:
    \begin{equation}
        \label{eq:brezis-proj}
        \text{for all }u\in\Tan_{T_t}\!\K,\qquad \langle v_t\circ T_t-w_t,u-w_t \rangle_{\Lrhoz}\leq0.
    \end{equation}
    This inequality, together with the fact that $\Tan_{T_t}\!\K$ is a cone, allows one to obtain both results as follows.
    $(i)$ By the stability of the cone $\Tan_{T_t}\!\K$ by nonnegative scalings, $2w_t$ belongs to $\Tan_{T_t}\!\K$.  Taking $u\coloneqq 0$ and $u\coloneqq 2w_t$ in \cref{eq:brezis-proj} therefore yields the two inequalities that constitute the desired equality.
    $(ii)$ Let $S\in\K$. Then it is immediate that $S-T_t$ belongs to $\Tan_{T_t}\!\K$. By convexity of $\Tan_{T_t}\!\K$, $\frac12(S-T_t+w_t)$ is in $\Tan_{T_t}\!\K$ as well; and the same goes for $S-T_t+w_t$ by stability of the cone by nonnegative scalings. Taking $u\coloneqq S-T_t+w_t$ in \cref{eq:brezis-proj} then gives the desired inequality.
\end{proof}
\begin{remark}[A variational characterization for $\partial_t T_t$]
\label{rem:variational}
Let us stress that the projection step in the constrained gradient flow \cref{eq:constrained-flow} can be written explicitly as
\begin{equation}
    \label{eq:constrained-flow-min}
    \partial_t T_t
         = {{\argmin_{w\in \Tan_{T_t}\!\!\K}} {\int_{\bR^d}}}\|v_t\circ T_t-w\|^2\dd\rho_0 \qquad \text{for a.e.~}t\in (0,t_{\text{max}}),
\end{equation}  
where $v_t\coloneqq -\nablaa D(T_t{}_*\rho_0)$.
This highlights that each time step in the constrained gradient flow \cref{eq:constrained-flow} is a \emph{quadratic minimization problem}.
If $\Tan_{T_t}\!\K$ contains tangent vectors of the form $\nabla\xi$ for $\xi \in C^2_c(\bR^d,\bR)$ (which is the case for instance if $T_t$ writes $T_t=\nabla\phi_t$ with $\phi_t$ strictly convex of $C^2$-regularity, see \cref{rem:tangent-cone}), one gets the following optimality condition
\begin{equation}
    \label{eq:optimality-condition-div}
    \div(\rho_0w_t)=\div(\rho_0v_t\circ T_t),
\end{equation}
where $w_t\coloneqq \partial_tT_t$; see \cref{sec:optimality-condition} for a proof.
This suggests interpreting the time variation $\partial_tT_t$ as the gradient component in the \nameref*{thm:helmholtz} of $v_t\circ T_t$ (see \cref{app:helmholtz} for a reminder). In sharp contrast to usual gradient flows on probability measures, this removal of the non-gradient component is performed
\emph{with respect to $\rho_0$}
and \emph{not} with respect to the current measure $\rho_t\coloneqq T_t{}_*\rho_0$.
\end{remark}

\subsection{Existence of solutions to the constrained gradient flow}
\label{sec:existence}

We assume that the functional $D:\PRd\to\bR$ is \emph{proper}---that is, its domain $\operatorname{Dom}(D)\coloneqq \{\rho\in\PRd\mid D(\rho)<\infty\}$ is non-empty---, and that it is \emph{bounded from below}.
In this section, we do not impose any other assumption on $D$; in particular, $D$ does not need to have $\gamma$ as a minimizer, nor to admit a minimizer at all.
Let $F=D\circ\pi$ be the lifted functional of $D$.

The main result of this section is \cref{thm:flow-defini}, which states the existence of a solution to the constrained gradient flow \cref{eq:constrained-flow}. To prove it, it will be convenient to consider the \emph{constrained functional} $F_{\K}\coloneqq F+\imath_\K$, where $\imath_\K(T)=0$ if $T\in \K$ and $\infty$ otherwise (the convex indicator function of $\K$).

\begin{lemma}[Element of minimal norm of the subdifferential of the constrained functional]
\label{lem:2flows}
Let $D:\PRd\to\bR$ be some functional
    that is
    Wasserstein differentiable,
    let $F=D\circ\pi$ be its lifted functional and $F_{\K}\coloneqq F+\imath_\K$.
Then for all $T\in\K$,
\begin{equation}
    \partial^\circ F_{\K}(T)
    =
    -\proj_{\Tan_T\K}(-\nablaa D(T_*\rho_0)\circ T).
\end{equation}
\end{lemma}
\begin{proof}
    By \cref{lem:frechet-diff}, since $D$ is Wasserstein differentiable, $F$ is Fréchet differentiable and therefore $\partial^\circ F=\nabla F$. 
    The sum rule for Fréchet subdifferentials \cite[Propositions 1.107 $(i)$ and 1.79]{mordukhovich2009variational} then gives that the Fréchet subdifferential of $F_{\K}$ at any $T\in \K$ is given by $\partial F_{\K}(T)=\nabla F(T)+\Nor_\K(T)$, where $\Nor_\K(T)$ is the normal cone of $\K$ at $T$. Its element of minimal norm is then
    \begin{multline}
        \partial^\circ F_{\K}(T)
        =\argmin_{v\in \partial F_{\K}}\|v\|_{\Lrhoz}
        =\nabla F(T)+\argmin_{n\in \Nor_T\K}\|\nabla F(T)+n\|_{\Lrhoz}
        \\
        =\nabla F(T)+\proj_{\Nor_T\K}(-\nabla F(T))
        =-\proj_{\Tan_T\K}(-\nabla F(T)),
    \end{multline}
    where we used the Moreau decomposition \cref{eq:moreau} for the closed convex cone $\Nor_T\K$.
    By \cref{lem:frechet-diff}, $\nabla F(T)=\nablaa D(T_*\rho_0)\circ T$ for any $T\in\K$, which yields the desired result.
\end{proof}

\noindent With this, we are ready to prove the main result of this section.
\begin{theorem}[Existence of solutions to the constrained gradient flow]
\label{thm:flow-defini}
Let $\rho_0\in\PRdac$ and let $D:\PRd\to\bR$ be some functional such that
\begin{equation}
    \label{eq:assumption-diff}
    \tag{\textsc{h}$_\lambda$}
    \begin{array}{l}
    \text{$D$ is l.s.c.~with respect to the weak topology on $\PRd$, Wasserstein differentiable,}\\\text{and $\lambda$-convex along generalized geodesics with anchor point $\rho_0$,}
    \end{array}
\end{equation}
with $\lambda\in\bR$.
Then, for every $t_{\text{max}}>0$, there exists a solution $(T_t)_t\in H^1([0,t_{\text{max}}],\K)$ to the constrained gradient flow \cref{eq:constrained-flow}.
\end{theorem}

\noindent 
In order to prove this theorem, it is sufficient by \cref{lem:2flows} to prove the existence of solutions to the Cauchy problem
\begin{equation}
    \label{eq:cauchy}
    \left\{\begin{aligned}
         &\ T_0=\id   \\
         &\ \partial_t T_t= -\partial^\circ F_{\K}(T_t)\qquad \text{for a.e.~}t\in (0,t_{\text{max}}).
    \end{aligned}\right.
\end{equation}
For this, we rely on classical results on the theory of generalized minimizing movements on Hilbert spaces \cite{rossi2006gradient} (see also \cite{de1993new,ambrosio2008gradient,muratori2020gradient} for a more general setting). 
For that purpose, two useful lemmas are \cref{lem:cvx-simple,cor:lsc-topo}, which allow to transfer convexity and lower semicontinuity of $D$ on $\PRd$ to convexity and lower semicontinuity of $F$ on $\LrhozRd$.
\begin{proof}
We use the \emph{(generalized) minimizing movement scheme} technique, which consists in approximating the gradient flow via a proximal gradient descent scheme.
Let $\tau>0$ be some time step, $\smash{\widehat T_0}\coloneqq\id\in \K$, and define for $k\geq0$ the following proximal step
\begin{equation}
    \label{eq:JKO}
    \tag{\textsc{prox}$_\tau$}
    \widehat T_{k+1}^\tau\in\argmin_{T\in \LrhozRd}F_{\K}(T)+\frac1{2\tau}\|T-\widehat T_k^\tau\|^2_{\Lrhoz}.
\end{equation}
Let us write $J_\tau:\LrhozRd\to\bR$ the functional being minimized in \cref{eq:JKO}.

    \noindent
    $(i)$ \emph{Well-posedness of the proximal step.}
    Let us fix $\smash{\widehat T_k^\tau}$ some element of $\K$ and show that the minimization step \cref{eq:JKO} is well-defined as long as $\tau$ is sufficiently small.
    Let us first note that $F$ is l.s.c.~(with respect to the strong topology) on $\LrhozRd$ by \cref{cor:lsc-topo}. 
    Because $\K$ is closed (\cref{prop:convex-closed}), $\imath_\K$ is lower semicontinuous (see \cite[Example 1.25]{bauschke2017convex}), hence $F_{\K}$ is as well, and then $J_\tau$ too.
    Now, since $D$ is $\lambda$-convex along generalized geodesics with anchor point $\rho_0$, $F$ is $\lambda$-convex in $\K$ (\cref{lem:cvx-simple}). Because $\K$ is convex (\cref{prop:convex-closed}), $\imath_\K$ is convex (see \cite[Example 8.3]{bauschke2017convex}), hence $F_{\K}$ is $\lambda$-convex, which in turns implies that $J_\tau$ is $(\tau^{-1}+\lambda)$-convex in $\LrhozRd$. Choosing $\tau$ small enough so that $\tau^{-1}+\lambda>0$, one can then apply \cite[Lemma 2.4.8]{ambrosio2008gradient} and finally get that $J_\tau$ admits a (unique) minimum, which belongs to $\operatorname{Dom}(F_{\K})\subset \K$.
    
    \noindent $(ii)$ \emph{Convergence when $\tau\to0$.}
    One can then apply \cite[Theorem~2]{rossi2006gradient} with the proper, lower semicontinuous and $\lambda$-convex functional $F_{\K}:\LrhozRd\to\bR$ on the Hilbert space $\LrhozRd$ to obtain the existence of a sequence $\tau_\ell\to0$ and corresponding solutions $\smash{\widehat T_k^{\tau_\ell}}$ that converge as $\ell\to\infty$ to a solution of \cref{eq:cauchy} of time regularity $H^1$, which concludes the proof.
\end{proof}

\begin{remark}[Functionals on $\PRd$ satisfying \cref{eq:assumption-diff}]
\label{rem:functionals}
The following differentiable functionals on $\PRd$ satisfy the assumptions \cref{eq:assumption-diff} of \cref{thm:flow-defini}.
\begin{itemize}[leftmargin=7mm]
    \item \emph{Relative entropy.} The relative entropy with respect to some $\lambda$-log-concave measure $\gamma$, where $\lambda\in\bR$, is l.s.c.~and $\lambda$-convex along generalized geodesics \cite[Theorem~9.4.11]{ambrosio2008gradient} \cite[Corollary~15.7]{ambrosio2021lectures}.
    \item \emph{Relative integral functionals.} 
    More generally, let $f:[0,\infty)\to[0,\infty]$ be a convex and l.s.c.~function such that $s\mapsto f(e^{-s})e^s$ is convex and nonincreasing in $(0,\infty)$, and $\gamma$ be some log-concave measure. Then the functional $D(\rho)=\int_{\bR^d}f(\dd\rho/\dd\gamma)\dd \gamma$ is l.s.c.~and convex along generalized geodesics in $\PRd$, under mild additional conditions on $f$ \cite[Theorem~9.4.12, Remark~9.3.8]{ambrosio2008gradient}.
        
    \item \emph{Potential energies.} Functionals of the form $D(\rho)=\int_{\bR^d}V\dd\rho$, where $V$ is proper, l.s.c.~and $\lambda$-convex on $\bR^d$, are l.s.c.~and $\lambda$-convex along generalized geodesics in $\PRd$ \cite[Proposition~9.3.2]{ambrosio2008gradient}.
    
    \item \emph{Interaction energies.} Functionals of the form $D(\rho)=\smash{\int_{(\bR^d)^k}}W\dd\rho^{\otimes k}$, where $W$ is proper, l.s.c.~and convex on $(\bR^d)^k$, are l.s.c.~and convex along generalized geodesics in $\PRd$ \cite[Proposition~9.3.5]{ambrosio2008gradient}.

    \item \emph{Entropic optimal transport.} The entropic regularization $\operatorname{OT}_\eps$ of the OT problem \cref{eq:kantorovitch} as well as the Sinkhorn divergence $\operatorname{Sk}_\eps$, are l.s.c.~\cite{feydy2019interpolating} and $\lambda$-convex for some $\lambda<0$ on compact domains \cite[Theorem~4.1]{carlier2024displacement}. The same goes for minus these two functionals. \qedhere
\end{itemize}
\end{remark}

\subsection{Convergence of the constrained gradient flow}
\label{sec:convergence}
The main result of this section is \cref{thm:convergence}, which states the convergence of the constrained gradient flow \cref{eq:constrained-flow} to the OT map $\optmap{\rho_\smallzero}{\gamma}$ under standard convexity assumptions on the functional $D$, with a convergence rate that does not depend on the ambient dimension $d$.
The proof of this result is similar to that of the standard convergence of gradient flows of functions that are star-convex around their minimizer on Hilbert spaces, with the slight but \emph{crucial} modification that here the time-update is given by \emph{a projection of} the gradient of $F$, and not the gradient itself.
We then instantiate these convergence results to the relative entropy (\cref{cor:entropy}), answering a question from \citeauthor{modin2016geometry}~\cite[Section 4.1.1]{modin2016geometry}.

In order to prove convergence of the constrained gradient flow in the following, we will need to ask for some convexity of $D$ along \emph{generalized geodesics with anchor point $\rho_0$ and endpoint $\gamma$}, that is, along curves of the form
\begin{equation}
    \label{eq:gen-geod-rho0-gamma}
    \rho_t=[(1-t)T+t\optmap{\rho_\smallzero}{\gamma}]_*\rho_0,
\end{equation}
where $T$ is any element of $\K$.
This assumption is different from the standard assumption of plain geodesic convexity in $\PRd$. 
Yet, note that both of these notions are implied by the more stringent---yet also standard---assumption of convexity along \emph{all generalized geodesics} \cref{eq:gen-geodesic-ac}.
Also note that there exist functionals on $\PRd$ that are convex along generalized geodesics of the form \cref{eq:gen-geod-rho0-gamma} and that are \emph{not} convex along \emph{all} generalized geodesics: for instance, the squared Wasserstein distance $\rho\mapsto \W_2(\rho,\rho_0)^2$ \cite[Remark 9.2.8]{ambrosio2008gradient}.
\cref{thm:convergence} below states the convergence result, which can be understood as follows: if $D$ is $\lambda$-convex along curves of the form \cref{eq:gen-geod-rho0-gamma} with $\lambda>0$, then the flow converges to the OT map; and if $D$ is merely convex along such curves, then the flow converges under the additional assumption of \emph{power-type growth} on $D$:
\begin{equation}
    \tag{\textsc{pg}}
    \label{eq:condition-QG}
    \text{there exist $c,\alpha>0$ such that }\quad \|\optmap{\rho_\smallzero}{\rho}-\optmap{\rho_\smallzero}{\gamma}\|^2_{\Lrhoz}\leq c \big(D(\rho)-D(\gamma)\big)^\alpha,
\end{equation}
which can be satisfied by the relative entropy under some conditions on $\rho_0$ (see \cref{lem:qg-entropy}).

\begin{theorem}[Convergence of the constrained gradient flow]
    \label{thm:convergence}
    Let $\rho_0\in\PRdac$.
    Suppose that $D$ admits a unique minimizer $\gamma$ in $\PRd$.
    Let $\optmap{\rho_\smallzero}{\gamma}$ be the OT map between $\rho_0$ and $\gamma$, and
    suppose that there exists a solution $(T_t)_t$ to the flow \cref{eq:constrained-flow}. 
     Then
    \begin{enumerate}[$(i)$,leftmargin=*]
        \item The map $t \mapsto D(T_t{}_*\rho_0)$ is nonincreasing.
        \item Suppose that $D$ is convex along curves of the form \cref{eq:gen-geod-rho0-gamma}. Then
        \begin{equation}
            \label{eq:result-cvg-F}
            \tag{$ii$.a}
            \text{for a.e.~$t\geq0$,}\qquad
            D(T_t{}_*\rho_0)-D(\gamma)\leq \smash{\frac1{2t}}\W_2(\rho_0,\gamma)^2,
        \end{equation}
        and if $D$ metrizes the weak convergence in $\PRd$, then $T_t\to \optmap{\rho_\smallzero}{\gamma}$ strongly in $\LrhozRd$ as $t\to\infty$.
        If additionally the power-type growth condition \cref{eq:condition-QG}
        is satisfied along the flow $(T_t)_t$, then
        \begin{equation}
            \label{eq:result-cvg-T}
            \tag{$ii$.b}
            \text{for a.e.~$t\geq0$,}\qquad
            \|T_t-\optmap{\rho_\smallzero}{\gamma}\|^2_{\Lrhoz}\leq \frac c{2^\alpha t^{\alpha}} \W_2(\rho_0,\gamma)^{2\alpha}.
        \end{equation}
        \item Suppose that $D$ is $\lambda$-convex
        along curves of the form \cref{eq:gen-geod-rho0-gamma},
        with $\lambda>0$.
        Then
        \begin{equation}
            \label{eq:result-cvg-F-strong}
            \tag{$iii$.a}
            \text{for a.e.~$t\geq0$,}\qquad
            D(T_t{}_*\rho_0)-D(\gamma)\leq e^{-2\lambda t}\big(D(\rho_0)-D(\gamma)\big)
        \end{equation}
        and
        \begin{equation}
            \label{eq:result-cvg-T-strong}
            \tag{$iii$.b}
            \text{for a.e.~$t\geq0$,}\qquad
            \|T_t-\optmap{\rho_\smallzero}{\gamma}\|^2_{\Lrhoz}\leq
            \frac4\lambda e^{-2\lambda t}\big(D(\rho_0)-D(\gamma)\big).
        \end{equation}
    \end{enumerate}
\end{theorem}
\noindent To prove the results, it is useful to see the constrained gradient flow as a gradient flow of the lifted functional $F$ on $\LrhozRd$, since proofs of convergence are easier in Hilbert spaces. 
Observe that
$D$ has a unique minimizer in $\PRd$ if and only if $F$ has a unique minimizer in $\K$ (see \cref{lem:minimizers}), and that if $D$ is convex along curves of the form \cref{eq:gen-geod-rho0-gamma}, then $F$ is star-convex around $\optmap{\rho_\smallzero}{\gamma}$ on $\K$ (see \cref{lem:cvx-simple}).
The proof of the convergence of the constrained gradient flow can therefore be done similarly to that of the standard convergence of gradient flows of functions that are star-convex around their minimizer on Hilbert spaces, except that the time-update is given by \emph{a projection of} the gradient of $F$, and not the gradient itself.
\begin{proof}
Let $v_t\coloneqq -\nablaa D(T_t{}_*\rho_0)$ and let $w_t\coloneqq \partial_tT_t$ be the solution of the constrained gradient flow \cref{eq:constrained-flow} at time $t$.
Recall that the gradient of $F$ at $T_t$ in the Hilbert space $\LrhozRd$ is given by
\begin{equation}
\nabla F(T_t)=\nablaa D(T_t{}_* \rho_0)\circ T_t =- v_t \circ T_t
\end{equation}
(see \cref{lem:frechet-diff}). 
Let us prove $(i)$, then $(iii)$, and finally $(ii)$.
\\
$(i)$ For a.e.~$t$,
        \begin{equation}
            \label{eq:decrease-f-T}
            \frac{\dd}{\dd t}\big(F(T_t)-F(\optmap{\rho_\smallzero}{\gamma})\big)
            =\langle \nabla F(T_t), \partial_t T_t\rangle_{\Lrhoz}
            =\langle -v_t\circ T_t, w_t\rangle_{\Lrhoz}
            \stackrel{\smash{(\star)}}= -\|w_t\|^2_{\Lrhoz}\leq0,
        \end{equation}
        where in $(\star)$ we applied \cref{lem:charac}, \cref{item:charac1}. This proves the result.
    \\
    $(iii)$ Assume that $D$ is $\lambda$-convex along curves of the form \cref{eq:gen-geod-rho0-gamma}, with $\lambda>0$. By \cref{lem:cvx-simple}, this means that $F$ is $\lambda$-star-convex around $\optmap{\rho_\smallzero}{\gamma}$ on $\K$, that is, for a.e.~$t$,
    \begin{equation}
        \label{eq:1st-order-conv}
        F(T_t)-F(\optmap{\rho_\smallzero}{\gamma})\leq \langle T_t-\optmap{\rho_\smallzero}{\gamma},-v_t\circ T_t\rangle_{\Lrhoz}-\frac\lambda2\|T_t-\optmap{\rho_\smallzero}{\gamma}\|^2_{\Lrhoz}
        \stackrel{\smash{(\star\star)}}=\langle T_t-\optmap{\rho_\smallzero}{\gamma},-w_t\rangle_{\Lrhoz}-\frac\lambda2\|T_t-\optmap{\rho_\smallzero}{\gamma}\|^2_{\Lrhoz},
    \end{equation}
    where in $(\star\star)$ we applied \cref{lem:charac}, \cref{item:charac2}.
        To show convergence of the values of $F$, one can apply Young's inequality to \cref{eq:1st-order-conv}:
        \begin{equation}
            \label{eq:PL-from-SC-T}
            F(T_t)-F(\optmap{\rho_\smallzero}{\gamma})
            \stackrel{\smash{\cref{eq:1st-order-conv}}}\leq \langle T_t-\optmap{\rho_\smallzero}{\gamma},-w_t\rangle_{\Lrhoz}-\frac\lambda2\|T_t-\optmap{\rho_\smallzero}{\gamma}\|^2_{\Lrhoz}
            \leq \frac{1}{2\lambda}\|w_t\|^2_{\Lrhoz},
        \end{equation}
        which, as a side note, is the Polyak--Łojasiewicz condition for $F$ constrained to $\K$ \cite{polyak1964gradient,lojasiewicz1963topological}.
        Then for a.e.~$t$,
        \begin{equation}
            \frac{\dd}{\dd t}\big(F(T_t)-F(\optmap{\rho_\smallzero}{\gamma})\big)
            \stackrel{\smash{\cref{eq:decrease-f-T}}}= -\|w_t\|^2_{\Lrhoz}
            \stackrel{\smash{\cref{eq:PL-from-SC-T}}}\leq -2\lambda\big(F(T_t)-F(\optmap{\rho_\smallzero}{\gamma})\big),
        \end{equation}
        and Grönwall's lemma then yields the exponential convergence of $F(T_t)$ to $F(\optmap{\rho_\smallzero}{\gamma})$ as desired for \cref{eq:result-cvg-F-strong}.
        To show convergence of $T_t$,
        let us use once again the $\lambda$-convexity of $F$, this time along the curve $(1-s)T_t+s\optmap{\rho_\smallzero}{\gamma}$:
        \begin{equation}
            F((1-s)T_t+s\optmap{\rho_\smallzero}{\gamma})\leq (1-s)F(T_t)+sF(\optmap{\rho_\smallzero}{\gamma})-\frac\lambda 2s(1-s)\|T_t-\optmap{\rho_\smallzero}{\gamma}\|^2_{\Lrhoz}.
        \end{equation}
        Evaluating at $s=\frac12$ and using the optimality of $\optmap{\rho_\smallzero}{\gamma}$ then yields for a.e.~$t$
        \begin{equation}
            0\leq F((T_t+\optmap{\rho_\smallzero}{\gamma})/2)-F(\optmap{\rho_\smallzero}{\gamma})\leq  \frac12\big(F(T_t)-F(\optmap{\rho_\smallzero}{\gamma})\big)-\frac\lambda 8\|T_t-\optmap{\rho_\smallzero}{\gamma}\|^2_{\Lrhoz},
        \end{equation}
        and finally 
        \begin{equation}
            \|T_t-\optmap{\rho_\smallzero}{\gamma}\|^2_{\Lrhoz}\leq \frac4\lambda \big(F(T_t)-F(\optmap{\rho_\smallzero}{\gamma})\big)\leq \frac4\lambda e^{-2\lambda t}\big(F(\id)-F(\optmap{\rho_\smallzero}{\gamma})\big),
        \end{equation}
        which gives \cref{eq:result-cvg-T-strong} and therefore completes the proof for $(iii)$.
        \\
        $(ii)$ Assume now that $D$ is merely convex along curves of the form \cref{eq:gen-geod-rho0-gamma}. By \cref{lem:cvx-simple}, this means that $F$ is star-convex around $\optmap{\rho_\smallzero}{\gamma}$ on $\K$, and \cref{eq:1st-order-conv} holds with $\lambda=0$. Consider then the Lyapunov function $V(t)=\frac12\|T_t-\optmap{\rho_\smallzero}{\gamma}\|^2_{\Lrhoz}$. Its time derivative (a.e.~in $t$) is
        \begin{equation}
            \label{eq:V-derivative-T}
            \frac{\dd}{\dd t}V(t)
            =\langle T_t-\optmap{\rho_\smallzero}{\gamma},w_t\rangle_{\Lrhoz}
            \stackrel{\smash{\cref{eq:1st-order-conv}}}\leq-\big(F(T_t)-F(\optmap{\rho_\smallzero}{\gamma})\big)\leq0,
        \end{equation}
        which ensures that $V$ is nonincreasing.
        Integrating over time yields for a.e.~$t$
        \begin{equation}
            t\big(F(T_t)-F(\optmap{\rho_\smallzero}{\gamma})\big)\leq\int_0^t\big(F(T_s)-F(\optmap{\rho_\smallzero}{\gamma})\big)\dd s\leq -\int_0^t\frac{\dd}{\dd s}V(s)\dd s=V(0)-V(t),
        \end{equation}
        and therefore
        \begin{equation}
            \label{eq:diffFconvex}
            F(T_t)-F(\optmap{\rho_\smallzero}{\gamma})\leq \frac1{2t}\big(\|\id-\optmap{\rho_\smallzero}{\gamma}\|^2_{\Lrhoz}-\|T_t-\optmap{\rho_\smallzero}{\gamma}\|^2_{\Lrhoz}\big)\leq\frac1{2t}\|\id-\optmap{\rho_\smallzero}{\gamma}\|^2_{\Lrhoz},
        \end{equation}
        which is the desired result \cref{eq:result-cvg-F}. If $D$ metrizes the weak convergence in $\PRd$, then $T_t{}_*\rho_0\rightharpoonup \gamma$, and by continuity of the mapping $\rho\mapsto \optmap{\rho}{\gamma}$
        (see e.g., \cite[Proposition~1.4]{letrouit2025lectures} for a detailed proof), $T_n\to \optmap{\rho_\smallzero}{\gamma}$ strongly in $\LrhozRd$.
        In the case where the condition \cref{eq:condition-QG} is satisfied, using \cref{eq:condition-QG} and the convergence of the values \cref{eq:diffFconvex} yields the desired \cref{eq:result-cvg-T}.
\end{proof}
\begin{lemma}[The relative entropy satisfies \cref{eq:condition-QG}]
    \label{lem:qg-entropy}
    Suppose that the support of $\rho_0$ is a John domain\footnote{Formally, a domain is a John domain \cite{john1961rotation,martio1979injectivity} if it is possible to move from one point to another while staying quantitatively away from the boundary. For instance, bounded domains with Lipschitz boundary or bounded convex sets are John domains. John domains are necessarily bounded. See \cite[Section 1.2]{letrouit2024gluing} and references therein for an account on John domains.}, that $\rho_0$ has a density that is bounded above and below by positive constants. 
    Then the relative entropy satisfies \cref{eq:condition-QG} with $\alpha=\frac16$ for all compactly supported $\rho$ and $\gamma$.
\end{lemma}
\begin{proof}
    Using Pinsker's inequality \cite{csiszar1963informationstheoretische,kullback1997information,pinsker1964information} and \cite[Particular case 6.16]{villani2009optimal} yields
    \begin{equation}
        H(\rho\midd\gamma)\geq 2\|\rho-\gamma\|_{\text{TV}} \geq\frac2{M^2}\W_1(\rho,\gamma)^2,
    \end{equation}
    where $M$ is an upper bound on the diameter of the supports of $\rho_0$ and $\gamma$.
    A recent result by \citeauthor{letrouit2024gluing}~\cite[Theorem~1.7]{letrouit2024gluing} then states that for all $\rho$ and $\gamma$ satisfying the conditions of \cref{lem:qg-entropy}\footnote{These assumptions can be relaxed. $(i)$ The compactness assumption for the supports of $\rho_t$ and $\gamma$ can be relaxed to finiteness of the $p$\textsuperscript{th}-order moment for some $p\in\bR$ if $p\geq 4$ and $p>d$, when replacing the exponent $\smash{\frac13}$ by an exponent $\smash{\frac p{3p+8d}}$ \cite[Corollary~4.4]{delalande2023quantitative}. $(ii)$ The boundedness assumption for $\rho_0$ can be relaxed to some control of the decay of $\rho_0$ when approaching the boundary of the domain, when replacing the exponent $\smash{\frac13}$ by an exponent $\smash{\frac13-\eta}$ for some $\eta>0$ \cite[Theorem~1.10]{letrouit2024gluing}.}, then   
    \begin{equation}
    \|\optmap{\rho_\smallzero}{\rho}-\optmap{\rho_\smallzero}{\gamma}\|_{\Lrhoz}^2\leq \tilde c \W_1(\rho,\gamma)^{1/3},
    \end{equation}
    hence the result.
\end{proof}

\begin{remark}[Other functionals satisfying \cref{eq:condition-QG}] \label{rem:qg_general}
    \cref{thm:convergence}, \cref{eq:result-cvg-T} requires two conditions on the functional $D$: the power-type growth condition \cref{eq:condition-QG}, and convexity along curves of the form \cref{eq:gen-geod-rho0-gamma}. The former is not very restrictive: using a similar argument as in \cref{lem:qg-entropy}, one can show that it is satisfied on bounded domains by the TV norm, the Hellinger distance, the flat norm \cite{hanin1992kantorovich}, the Dudley metric \cite{dudley1945real}, or MMD functionals with Sobolev kernel of regularity $\smash{s\geq \frac d2+1}$ (or even more general kernels, see \cite{fiedler2023lipschitz}). The latter, however, is far more stringent: few functionals are known to be convex apart from those mentioned in \cref{rem:functionals}.
\end{remark}
\begin{remark}[When does $D$ metrize weak convergence?]
    For $D$ to metrize the weak convergence in the space $\PRd$ of probability measures with finite second-order moment, it is for instance sufficient that it bounds the $\W_2$ distance, which is true whenever \cref{eq:condition-QG} is satisfied.
    On bounded domains, the weak and the narrow topologies on $\PRd$ coincide \cite[Corollary~6.13]{villani2009optimal}, and one might come up with other functionals on $\PRd$ that metrize this topology (e.g., integral probability metrics \cite{fortet1953convergence,muller1997integral,sriperumbudur2009integral} or MMD functionals \cite{simon2023metrizing}). 
\end{remark}

\noindent Let us now instantiate the results of \cref{thm:flow-defini,thm:convergence} to the relative entropy.
\begin{corollary}[Constrained gradient flow for the relative entropy]
    \label{cor:entropy}
    Let $D:\rho\mapsto H(\rho\midd\gamma)$ be the relative entropy with respect to some $\lambda$-log-concave measure $\gamma\in\PRd$, where $\lambda\in\bR$. Then the constrained gradient flow \cref{eq:constrained-flow} admits a solution $(T_t)_t$. Moreover, we have the following:
    \begin{enumerate}[$(i)$,leftmargin=*]
        \item Assume $\lambda=0$. Then, as $t\to\infty$, $H(T_t{}_*\rho_0\midd\gamma)\to0$ with convergence rate $O(t^{-1})$.
        If additionally the assumptions of \cref{lem:qg-entropy} are satisfied, then $T_t\to\optmap{\rho_\smallzero}{\gamma}$ strongly in $\LrhozRd$ with convergence rate $O(t^{-1/6})$.
        \item Assume $\lambda>0$. Then, as $t\to\infty$, $H(T_t{}_*\rho_0\midd\gamma)\to0$ with convergence rate $O(e^{-2\lambda t})$ and $T_t\to\optmap{\rho_\smallzero}{\gamma}$ strongly in $\LrhozRd$ with convergence rate $O(e^{-2\lambda t})$.\qedhere
    \end{enumerate}
\end{corollary}
\begin{proof}
The relative entropy is ($\lambda$-)convex along generalized geodesics in $\PRd$ if and only if $\gamma$ is ($\lambda$-)log-concave \cite[Theorem~9.4.11]{ambrosio2008gradient}.
As mentioned in \cref{rem:functionals}, it satisfies \cref{eq:assumption-diff}, which allows to apply \cref{thm:flow-defini} and get the existence of a solution to \cref{eq:constrained-flow}.
In the $\lambda$-convex case $(ii)$ with $\lambda>0$, \cref{thm:convergence} directly gives the result. In the merely convex case $(i)$ with the additional assumptions of \cref{lem:qg-entropy}, the relative entropy satisfies assumption \cref{eq:condition-QG} with $\alpha=1/6$ and \cref{thm:convergence} then gives the desired convergence results.
\end{proof}

\section{Gradient descent for parameterized OT maps}
\label{sec:guided}

The theoretical results derived in \cref{sec:theory} show that an OT map between an initial measure $\rho_0 \in \PRdac$ and a target measure $\gamma\in\PRd$ can be obtained as the infinite-time limit of the constrained gradient flow \cref{eq:constrained-flow} of some suitable functional $T \mapsto D(T_* \rho_0)$ in $\K$ (e.g., the relative entropy when $\gamma$ is strongly log-concave, see \cref{cor:entropy}). 
Turning this theoretical result into a practical algorithm requires the following two steps.
\begin{enumerate}[$(i)$,leftmargin=*]
    \item \emph{Discretizing the flow in time}, which makes it a gradient \emph{descent}.
    This comes in two flavors:
    either \emph{explicitly}, discretizing the variational characterization \cref{eq:constrained-flow-min}, yielding
    \begin{equation}
        \label{eq:gd-explicit-T}
        \widehat T_{k+1}^\tau
             \in \argmin_{T\in\K} {\int_{\bR^d}}\Big\|-\nablaa D(\widehat T_{k}^\tau{}_*\rho_0)\circ \widehat T_{k}^\tau-\frac{T-\widehat T_{k}^\tau}\tau\Big\|^2\dd\rho_0,
    \end{equation}  
    or \emph{implicitly}, using the proximal scheme \cref{eq:JKO}
    \begin{equation}
        \label{eq:gd-implicit-T}
        \widehat T_{k+1}^\tau\in\argmin_{T\in\K}D(T_*{}\rho_0)+\frac1{2\tau}\|T-\widehat T_{k}^\tau\|^2_{\Lrhoz},
    \end{equation}
    where $\smash{\widehat T_0^\tau}\coloneqq\id\in \K$ and where $\tau>0$ is some time step.
    \item \emph{Replacing the set $\K$ by a parameterization} over a finite-dimensional convex set $\Theta \subset \bR^m$, that is, by a set $\{T_\theta \mid\theta \in \Theta\} \subset \K$.
    The parameterization can be handled as $\theta\mapsto \nabla \phi_\theta$, where $\phi_\theta : \bR^d \to \bR$ is a parameterized convex function, typically an Input Convex Neural Network (ICNN) or Log-Sum-Exp (LSE) network, where $\theta$ denotes the network's parameters.
\end{enumerate}
\cref{sec:unparam} focuses on the convergence properties of the \emph{time discretization} $(i)$ \emph{of the flow} (showing that the implicit discrete scheme converges to the OT map as $k\to\infty$, and that one recovers the (time-continuous) constrained gradient flow when $\tau\to0$).
\cref{sec:procedure} then integrates the \emph{parameterization} $(ii)$ \emph{of $\K$}, yielding implementable schemes. 
Those schemes are then showed in \cref{sec:natural-gd} to belong to the class of \emph{natural gradient} schemes, which sheds light on their good computational behavior, exposed in \cref{sec:numerical}.

\subsection{From gradient flow to gradient descent}
\label{sec:unparam}
In this section, the functional $D:\PRd\to\bR$ will need to satisfy condition \cref{eq:assumption-diff}, which consists in being l.s.c.~with respect to the weak topology on $\PRd$, Wasserstein differentiability, and $\lambda$-convexity along generalized geodesics with anchor point $\rho_0$ for some $\lambda\in\bR$.
Let $F=D\circ\pi$ be the lifted functional of $D$.

The two next propositions below focus on the convergence properties of the implicit scheme \cref{eq:gd-implicit-T}; see \cref{rem:explicit-gd} for a discussion on why one cannot expect the same properties for the explicit scheme without additional assumptions on $D$.
First, we show that the implicit scheme converges to the OT map in the infinite-time limit.
\begin{proposition}[Convergence of the proximal scheme to the OT map]
    \label{prop:conv-tau-fixed}
    Let $D:\PRd\to\bR$ be some functional satisfying \cref{eq:assumption-diff} with $\lambda\in\bR$ and with unique minimizer $\gamma$, let $\tau>0$, and let
    $\smash{(\widehat T_{k}^\tau)_k}$ be a solution of \cref{eq:gd-implicit-T}. Then
    \begin{enumerate}[$(i)$,leftmargin=*]
        \item if $\lambda=0$, then ${\widehat T_{k}^\tau\rightharpoonup\optmap{\rho_\smallzero}{\gamma}}$ weakly in $\LrhozRd$ as $k\to\infty$;
        \item if $\lambda>0$, then ${\widehat T_{k}^\tau\to\optmap{\rho_\smallzero}{\gamma}}$ strongly in $\LrhozRd$ as $k\to\infty$, with convergence rate $O(\alpha^k)$ where $\alpha=(1+\lambda^2\tau^2/4)^{-1/2}<1$.
    \end{enumerate}
\end{proposition}
\begin{proof}
$F$ is $\lambda$-convex (\cref{lem:cvx-simple}) and l.s.c.~(\cref{cor:lsc-topo}) on $\K$. 
Since $\K$ is convex and closed in $\LrhozRd$ (\cref{prop:convex-closed}), the indicator function $\imath_{\K}$ is convex and l.s.c.~\cite[Examples 1.25 and 8.3]{bauschke2017convex}, and $F+\imath_{\K}$ is therefore $\lambda$-convex and l.s.c.
In case $(i)$ where $\lambda=0$, one can therefore apply \cite[Theorem~1]{rockafellar1976monotone} to obtain the weak convergence to the unique minimizer $\optmap{\rho_\smallzero}{\gamma}$ of $F$ in $\K$;
and in case $(ii)$ where $\lambda>0$, one might apply \cite[Theorem~2]{rockafellar1976monotone} to obtain the strong convergence. The desired convergence rate can be obtained by combining \cite[Theorem~2, Proposition~7, and Remark~4]{rockafellar1976monotone} with the $\lambda$-convex function $F+\imath_{\K}$.
\end{proof}
\begin{remark}[Inexact solving of \cref{eq:gd-implicit-T}]
    \label{rem:errors}
    It is worth mentioning that \cite[Theorems 1 and 2]{rockafellar1976monotone} allows \cref{prop:conv-tau-fixed} to hold even when the solving of \cref{eq:gd-implicit-T} is not exact, as long as the successive errors are small enough. Namely, writing $J_k^\tau$ the functional to be minimized in \cref{eq:gd-implicit-T}, \cref{prop:conv-tau-fixed} still holds if there exists a sequence $(\delta_k)_k$ such that $\sum_{k=0}^{\infty}\delta_k<\infty$ and
    \begin{equation}
        \label{eq:condition1}
        d_{\Lrhoz}\big(\widehat T_{k+1}^\tau,\argmin_{\K}J_k^\tau\big)\leq \delta_k\|\widehat T_{k+1}^\tau-\widehat T_{k}^\tau\|_{\Lrhoz}\qquad \text{for all }k\geq0.
    \end{equation}
    In that case, the convergence rate in $(ii)$ becomes $O(\Pi_{\ell=0}^k\frac{\alpha+\delta_\ell}{1-\delta_\ell})$.
    Property \cref{eq:condition1} is hard to check numerically---luckily, it is implied \cite[Section 1]{rockafellar1976monotone} by the weaker condition
    \begin{equation}
        \label{eq:condition2}
        \big\|\partial F(\widehat T_{k+1}^\tau)+\frac1{2\tau}(\widehat T_{k+1}^\tau-\widehat T_{k}^\tau)\big\|_{\Lrhoz}\leq \frac{\delta_k}{\tau}\|\widehat T_{k+1}^\tau-\widehat T_{k}^\tau\|_{\Lrhoz}\qquad \text{for all }k\geq0,
    \end{equation}
    a quantity which is easier to compute.
\end{remark}
From the proof of \cref{thm:flow-defini}, we also get that in the limit $\tau\to0$, the implicit scheme \cref{eq:gd-implicit-T} converges to the time-continuous constrained gradient flow, in the following sense. 
\begin{proposition}[Convergence of the proximal scheme to the constrained gradient flow when $\tau\to0$]
    \label{prop:conv-to-continuous}
Let $D:\PRd\to\bR$ be some functional satisfying \cref{eq:assumption-diff} with $\lambda\in\bR$.
    Let $\smash{\widehat T_{k}^\tau}$ be a family of solutions of \cref{eq:gd-implicit-T} indexed by $\tau>0$ and $\smash{\widehat T^\tau}$ their associated piecewise constant interpolations, defined as $\smash{\widehat T_\tau(t)}\coloneqq \smash{\widehat T_{k}^\tau}$ for $t\in ((k-1)\tau,k\tau]$.
    Then for every $t_{\text{max}}>0$, there exists
    a sequence $\tau_k\to0$ and a solution $(T_t)_t\in H^1([0,t_{\text{max}}],\K)$ to \cref{eq:constrained-flow} such that
    \begin{equation}
    \smash{\widehat T_{\tau_k}(t)\xrightarrow[k\to\infty]{} T_t} \qquad \text{for a.e.~} t\in[0,t_{\text{max}}].
    \end{equation}
\end{proposition}

Of importance for our numerical study, note that \cref{prop:conv-tau-fixed,prop:conv-to-continuous} above hold for the relative entropy with respect to some $\lambda$-log-concave measure $\gamma$, where $\lambda>0$.

\begin{remark}[Convergence results for the explicit scheme]
    \label{rem:explicit-gd}
    It is worth mentioning that one cannot hope for the convergence results of \cref{prop:conv-tau-fixed,prop:conv-to-continuous} for the \emph{explicit} scheme \cref{eq:gd-explicit-T} without assuming some smoothness on the functional $D$---typically, some Lipschitz continuity on its gradient. Unfortunately, this smoothness assumption does not hold for the relative entropy, our functional of choice in this work, and we do not know of any other functional that would satisfy the assumptions of \cref{thm:flow-defini,thm:convergence} while also being smooth. 
    In the next sections, we implement the explicit scheme anyway, hoping that the parameterization $\theta\mapsto T_\theta$ by neural networks induces some smoothness on the optimized functional thanks to an architectural regularization.
\end{remark}

\subsection{Gradient descent for parameterized OT maps}
\label{sec:procedure}
In this section, we write down the time-discrete schemes \cref{eq:gd-explicit-T,eq:gd-implicit-T} in the context of a parameterization $\theta\mapsto T_\theta\in\K$, switching from an optimization on the set $\K$ of OT maps to some parameter space $\Theta\subset\bR^m$. This parameterization takes the form $\theta\mapsto\nabla\phi_\theta$, where $\phi_\theta : \bR^d \to \bR$ is a parameterized convex function, typically an Input Convex Neural Network (ICNN) or Log-Sum-Exp (LSE) network, and where $\theta\in\Theta$ denotes the network's parameters. Good expressivity properties of such architectures (see \cite{chen2018optimal,gagneux2025convexity} for ICNNs and \cite{calafiore2019log,calafiore2020universal} for LSE networks) suggest that they may provide a good approximation of $\K$ when their size is sufficiently big.

With such a parameterization, the scheme \cref{eq:gd-explicit-T} in $\K$ becomes
the \emph{explicit scheme} in $\Theta$
\begin{equation}
    \tag{\textsc{gd}, expl.}
    \label{eq:explicit_scheme}
    \theta_{k+1}\in\argmin_{\theta\in\Theta} \int_{\bR^d}\Big\|-\nablaa D(T_{\theta_k}{}_*\rho_0)\circ T_{\theta_k}-\frac{T_{\theta}-T_{\theta_k}}\tau\Big\|^2 \dd\rho_0,
\end{equation}
and the scheme \cref{eq:gd-implicit-T} in $\K$ becomes the \emph{implicit scheme} in $\Theta$
\begin{equation}
    \tag{\textsc{gd}, impl.}
    \label{eq:implicit_scheme}
    \theta_{k+1} \in \argmin_{\theta \in \Theta} D(T_\theta{}_*\rho_0) + \frac{1}{2\tau} \| T_\theta - T_{\theta_k}\|^2_{\Lrhoz},
\end{equation}
where both schemes start from some initial parameter $\theta_0 \in \Theta$ and where $\tau > 0$ is some fixed step size. 
The performance of these two schemes will have to be compared with that of the standard Euclidean gradient descent in the parameter space $\Theta$, that is,
\begin{equation}
    \label{eq:eucl-gd}
    \tag{\textsc{e}ucl.\textsc{gd}}
    \theta_{k+1} = \theta_k - \tau \nabla_\theta D(T_{\theta_k}{}_*\rho_0),
\end{equation}
which is the (explicit) time-discretization of the Euclidean gradient flow
\begin{equation}
    \label{eq:eucl-gf}
    \tag{\textsc{e}ucl.\textsc{gf}}
    \partial_t\theta_t=-\nabla_\theta D(T_{\theta_t}{}_*\rho_0).
\end{equation}
While the Euclidean gradient flow attempts to minimize $\theta \mapsto D({T_\theta}_* \rho_0)$ by following the steepest descent direction \emph{in the parameter space}, the flow \cref{eq:constrained-flow} we consider uses the information of descent \emph{in $\LrhozRd$ directly}. This remark is at the heart of the next section, which provides intuition on why such a behavior is well-suited for convergence.

\subsection{Link with natural gradient flows}
\label{sec:natural-gd}
Before focusing on the implementation details, we provide in this subsection a geometric interpretation of the dynamic on $\Theta$ induced by our schemes. 
We show that it can be seen as a gradient flow in $\Theta$ endowed not with the standard Euclidean metric but with the \emph{pullback} metric of the flat $\Lrhoz\!$-metric by the mapping $\theta\mapsto T_\theta$. 
This procedure is known under the name of \emph{natural gradient flow} (or \emph{natural gradient descent} for its discrete counterpart in the machine learning literature \cite{bai2022generalized,zhang2019fast}) and takes origins in the seminal work of \cite{amari1998natural} that pulled back the Fisher--Rao metric from $\PRd$ to $\Theta$. 
This observation sheds light on the good computational behavior of our constrained gradient flow (which we detail in \cref{sec:numerical} below):
whereas the performance of \cref{eq:eucl-gf} strongly depends on the parameterization $\theta\mapsto T_\theta$ \cite{martens2010deep,sutskever2013importance}, the natural gradient flows have good invariance properties with respect to re-parameterizations \cite{arbel2019kernelized,van2023invariance}.

Let us recall that $\theta\mapsto T_\theta$ is a parameterization of the set $\K$ of OT maps (encoded as gradients of convex functions). Observe that in the continuous time limit ($\tau\to0$), the descent step \cref{eq:explicit_scheme} formally yields the following evolution equation:
\begin{equation}
    \label{eq:guided-flow}
    \tag{\textsc{n}at.\textsc{gf}}
    \partial_t\theta_t\in\argmin_{\delta\theta\in \Tan_{\theta_t}\!\!\Theta} \int_{\bR^d}\big\|-\nablaa D(T_{\theta_t}{}_*\rho_0)\circ T_{\theta_t}-\nabla_\theta T_{\theta_t}.\delta\theta\big\|^2 \dd\rho_0,
\end{equation}
starting from some initial parameter $\theta_0\in\Theta$.
Under the additional assumption that the matrix $\smash{\int_{\bR^d}(\nabla_\theta T_{\theta_t})^\top\nabla_\theta T_{\theta_t}\dd\rho_0}$ is invertible, the optimality equation associated to \cref{eq:guided-flow} reads:
\begin{equation}
    \label{eq:guided-explicit}
    \partial_t\theta_t
    =-\Big[\int_{\bR^d}(\nabla_\theta T_{\theta_t})^\top\nabla_\theta T_{\theta_t}\dd\rho_0\Big]^{-1}\int_{\bR^d}(\nabla_\theta T_{\theta_t} )^\top \nablaa D(T_{\theta_t}{}_*\rho_0)\circ T_{\theta_t}\dd\rho_0.
\end{equation}    
Using the chain rule in \cref{eq:eucl-gf} makes explicit that the only (but crucial) difference between the standard gradient flow \cref{eq:eucl-gf} and the flow \cref{eq:guided-flow} is a preconditioning matrix, which is the inverse of $\smash{\int_{\bR^d}(\nabla_\theta T_{\theta_t})^\top\nabla_\theta T_{\theta_t}\dd\rho_0}$ (which is sometimes called \emph{neural tangent kernel} \cite{jacot2018neural,bai2022generalized}). This hints at a connection between \cref{eq:guided-flow} and the \emph{natural gradient} schemes, which we detail below.
As a computational side note, the numerical complexity of inverting this $m\times m$ matrix prevents the direct use of \cref{eq:guided-explicit} as an explicit scheme, and the method of choice consists of solving the optimization problem \cref{eq:guided-flow} (see \cref{sec:numerical} for the implementation details).

Let us now give the general definition of natural gradient flows, instantiate it to our setting---that is, in the space of transport maps $\LrhozRd$ endowed with its (flat) Hilbert metric---, and show that \cref{eq:guided-flow} fits this framework.

\begin{definition}[Natural gradient flow]
    \label{def:natural-gd-general}
    Let $\Theta$ be a finite-dimensional manifold and $M$ be a (possibly infinite-dimensional) Riemannian manifold with metric $g$.
    Let $\sigma$, $F$ and $L$ be defined as in the following sequence of mappings:
    \begin{equation}
    \label{diagram:natural}
    \begin{tikzcd}
    L:\Theta \ar[r, "\sigma"] & M \ar[r, "F"] & \bR,
    \end{tikzcd}
    \end{equation}
    that is, $L:\theta\mapsto F(\sigma_\theta)$. Assume that $\sigma$ and $F$ are differentiable, and that $d_\theta \sigma$ is injective for all $\theta\in\Theta$\footnote{The assumption of injectivity of the differential of $\sigma$ is required for the metric $\sigma^*g$ to be nondegenerate on $\Theta$. See \cite{van2023invariance,bai2022generalized} for generalizations of the natural gradient to cases where the differential of $\theta\mapsto \rho_\theta$ is allowed to be singular.}.
    Then the \emph{natural gradient flow} of $F$ on $\Theta$ is defined as the gradient flow of $\theta\mapsto F(\sigma_\theta)$ on $\Theta$ with respect to the pullback metric of $g$ by $\sigma$, which we note $\sigma^*g$ and which is defined as
    \begin{equation}
    (\sigma^*g)_\theta(\delta\theta,\delta\theta)\coloneqq g_{\sigma_\theta}(d_\theta \sigma[\delta\theta],d_\theta \sigma[\delta\theta])
    \qquad
    \text{for any $\theta\in\Theta$ and $\delta\theta\in T_\theta\Theta$.}
    \end{equation}
\end{definition}
\noindent Let us now instantiate this definition in the case where $M$ is the space $\LrhozRd$ of transport maps, for some convex parameter space $\Theta\subset\bR^m$.
\begin{definition}[$\Lrhoz\!$-natural gradient flow]
    Let $\Theta\subset\bR^m$. Let $\theta\mapsto T_\theta\in\LrhozRd$ be differentiable and such that $d_\theta T_\theta$ is injective for all $\theta\in\Theta$, and let $F:\LrhozRd\to\bR$ be differentiable.
    Then the \emph{$\Lrhoz\!$-natural gradient flow} of $F$ on $\Theta$ is the gradient flow of $\theta\mapsto F(T_\theta)$ on $\Theta$ with respect to the pullback metric of the flat $\Lrhoz\!$-metric by the map $\theta\mapsto T_\theta$, that is,
    \begin{equation}
    h_\theta(\delta\theta,\delta\theta)=\int_{\bR^d}\|d_\theta T_\theta[\delta\theta]\|^2\dd\rho_0
    \qquad
    \text{for any $\theta\in\Theta$ and $\delta\theta\in T_\theta\Theta$.}
    \end{equation}
\end{definition}

We now show that the parameterized constrained gradient flow \cref{eq:guided-flow} can be seen as a $\Lrhoz\!$-natural gradient flow on $\Theta$.
This is proved in \cref{cor:natural-gf}, which is a consequence of the following general lemma whose proof can be found in \cref{sec:natural-gd-proof}.

\def\naturalgdtitle{Natural gradient via quadratic minimization}
\def\naturalgdall{
    Let $\Theta$ be a finite-dimensional manifold and $M$ be a (possibly infinite-dimensional) Riemannian manifold with metric $g$.
    Let $\sigma:\Theta\to M$, $F:M\to\bR$ and $L=F\circ\sigma$,
    as in \cref{diagram:natural}.
    Assume that $\sigma$ and $F$ are differentiable, and that $d_\theta \sigma$ is injective for all $\theta\in\Theta$. Then
}

\begin{lemma}[\naturalgdtitle]
\label{prop:natural-gd}
\naturalgdall
    \begin{equation}
    \label{eq:minimization-general}
        \argmin_{\delta\theta\in T_\theta\Theta}\,\big\|\!\grad^g_MF(\sigma_\theta)-d_\theta \sigma[\delta\theta]\big\|_{g}^2
    \end{equation}
    is unique and equal to $\grad^{\sigma^*g}_\Theta L(\theta)$, that is, the gradient of $L$ with respect to the pullback metric $\sigma^*g$.
\end{lemma}
\begin{corollary}[The parameterized constrained gradient flow is a natural gradient flow]
    \label{cor:natural-gf}
    Let $\Theta\subset\bR^m$. Let $\theta\mapsto T_\theta\in\K$ be differentiable and such that $d_\theta T_\theta$ is injective for all $\theta\in\Theta$, and let $D:\PRd\to\bR$ be differentiable.
    Then then flow \cref{eq:guided-flow} is the $\Lrhoz\!$-natural gradient flow of $F:T\mapsto D(T_*\rho_0)$ on $\Theta$.
\end{corollary}
\begin{proof}
Taking $M=\LrhozRd$ and $\sigma:\theta\mapsto T_\theta$ in \cref{eq:minimization-general} and recalling that the gradient of $F$ at some $T\in\LrhozRd$ is given by $\nabla F(T)=\nablaa D(T_*\rho_0)\circ T$ whenever $D$ is differentiable (see \cref{lem:frechet-diff}),
one recovers the minimization problem in \cref{eq:guided-flow}.
The flow \cref{eq:guided-flow} is therefore the $\Lrhoz\!$-natural gradient flow of $\theta\mapsto F(T_\theta)=D(T_\theta{}_*\rho_0)$ with respect to the mapping $\theta\mapsto T_\theta$.
\end{proof}

As such, whereas the standard Euclidean gradient flow \cref{eq:eucl-gf} imposes a flat metric on the parameter space $\Theta$ and yields an evolution in a curved $\LrhozRd$ (endowed with the so-called \emph{neural tangent kernel} geometry \cite{jacot2018neural,bai2022generalized}), the $\Lrhoz\!$-natural gradient flow imposes a simpler geometry on $\LrhozRd$---that is, its flat Hilbert structure. 
Because many functionals are well-behaved in $\PRd$ with the Wasserstein metric (such as the relative entropy, $\lambda$-convex whenever the reference measure is $\lambda$-log-concave) and since this space is strongly linked to $\LrhozRd$ (the pushforward mapping $T\mapsto T_*\rho_0$ can be seen as an informal Riemannian submersion between those spaces \cite{otto2001geometry,modin2016geometry}), we believe that this pullback geometry on $\Theta$ is better suited for guaranteeing convergence of flows taking place on $\LrhozRd$, which our proof-of-concept experiments in \cref{sec:numerical} seem to confirm. See \cref{fig:natural-gd} below for a visual illustration.

\begin{figure}[!h]
\centering
\begin{subfigure}{0.45\textwidth}
    \includegraphics[width=\textwidth]{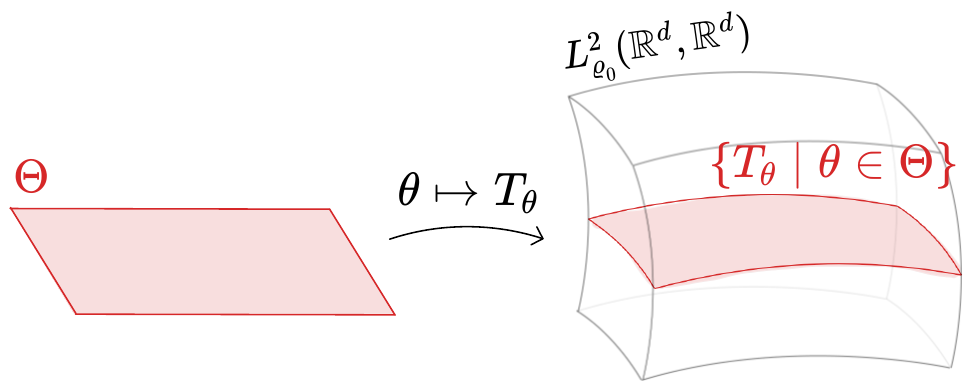}
    \label{fig:1a}
    \vspace{-5mm}
    \caption{Euclidean gradient structure in $\Theta$.}
\end{subfigure}
\hfill
\begin{subfigure}{0.45\textwidth}
    \includegraphics[width=\textwidth]{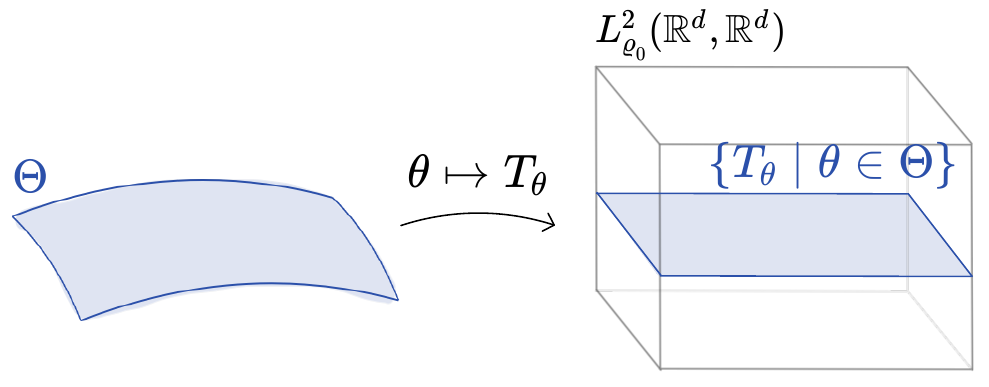}
    \vspace{-5mm}
    \caption{Natural gradient structure in $\Theta$.}
    \label{fig:1b}
\end{subfigure}
\caption{Simplified view of the geometries underlying the Euclidean \cref{eq:eucl-gf} and natural \cref{eq:guided-flow} gradient structures in $\Theta$. {\small (\textsc{a})} For the standard gradient flow, the parameter space $\Theta$ is endowed with the Euclidean metric and the optimization takes place in a curved $\LrhozRd$. {\small (\textsc{b})} For the $\Lrhoz\!$-natural gradient flow, the optimization takes place in a flat $\LrhozRd$ and $\Theta$ is endowed with the (non-flat) pullback metric.}
\label{fig:natural-gd}
\end{figure}

\begin{remark}[Wasserstein natural gradient flows]
    Another possible instantiation of \cref{def:natural-gd-general} is with a parameterization $\sigma:\theta\mapsto\rho_\theta$ of the set of probability measure $\PRd$. Whenever $\PRd$ is endowed with the Wasserstein(--Otto) metric, this procedure is called \emph{Wasserstein natural gradient flow} \cite{li2018natural,arbel2019kernelized}.
    It is worth mentioning that
    the parameterized constrained gradient flow \cref{eq:guided-flow} is \emph{not} a {Wasserstein} natural gradient flow with respect to $\theta\mapsto T_\theta{}_*\rho_0$. Indeed, in that case,
    the chain rule and the definition of the Wasserstein(--Otto) metric (see \cref{app:wasserstein-metric}) yield that \cref{eq:minimization-general} becomes
    \begin{equation}
        \label{eq:min-wass-ngf}
        \argmin_{\delta\theta\in T_\theta\Theta}\int_{\bR^d} \big\|-\nablaa D(T_{\theta}{}_*\rho_0)\circ T_{\theta}-\Pi^\nabla_{\rho_\theta}(\nabla_\theta T_{\theta}.\delta\theta\circ T^{-1}_\theta)\circ T_{\theta}\big\|\dd\rho_0,
    \end{equation}
    where $\Pi^\nabla_{\rho_\theta}$ is the operator that returns the gradient part in the \nameref*{thm:helmholtz} with respect to $\rho_\theta\coloneqq T_\theta{}_*\rho_0$ (see \cref{app:helmholtz} for reminders on this decomposition).
    As an alternative to \cref{eq:guided-flow}, this flow would be interesting to study; yet, from a computational perspective, it seems less convenient to implement, hence our choice to stick with \cref{eq:guided-flow} in this work.
\end{remark}

\begin{remark}[Unconstrained parameterization and drifting models]
    \label{rem:unconstrained}
    It is worth mentioning that the whole content of this section does not depend on the image of the parameterization $\theta\mapsto T_\theta$; in particular, everything holds if $\theta\mapsto T_\theta$ is a parameterization of (a subset of) the whole space $\LrhozRd$ of transport maps (not necessarily optimal). This is akin to the setting of drifting generative models (see \cref{sec:relatedworks}), hinting at their link with natural gradient descent schemes.
\end{remark}

\subsection{Implementation and numerical illustration}
\label{sec:numerical}

In this section, we present a practical implementation of the explicit and implicit schemes \cref{eq:explicit_scheme,eq:implicit_scheme} introduced in \cref{sec:procedure}, in the case where the functional $D$ of interest is the relative entropy\footnote{Note that both schemes can be used with any functional $D$, as long as one can numerically compute (an approximation of) its Wasserstein gradient, as we do in \cref{sec:practical} for the relative entropy.} $D = H(\cdot \midd \gamma)$ with respect to some strongly log-concave measure $\gamma$, which we write $\gamma \propto e^{-V}$ for some known strongly convex potential $V : \bR^d \to \bR$.

\noindent We recall that OT maps $T_\theta$ are parameterized as gradients of neural networks that are convex with respect to their input $x \in \bR^d$ (e.g., ICNNs),
where $\theta \in \Theta \subset \bR^m$ represent the network's weight parameters. 
\cref{alg:trad-train} below implements the standard Euclidean gradient descent \cref{eq:eucl-gd} on $\Theta$, while \cref{alg:new-train} implements the schemes \cref{eq:explicit_scheme,eq:implicit_scheme}, which discretize the constrained gradient flow---and which, under the light of \cref{sec:natural-gd}, can be interpreted as gradient descents on $\Theta$ for a different geometry.
\begin{figure}[!h]
\centering
\begin{minipage}[t]{.48\linewidth}
    \begin{algorithm}[H]
    \flushleft
    \caption{Euclidean gradient descent}
    \vspace{1mm}
    \textbf{Inputs:} source measure $\rho_0 \in \PRd$, initial parameter $\theta_0 \in \bR^m$, step size $\tau > 0$ \\
    \begin{algorithmic}[0]
        \For{$k\in[\![0,K-1]\!]$}
            \State let $\delta\theta \coloneqq -\nabla_\theta H(T_{\theta_k}{}_*\rho_0\midd \gamma)$
            \State update $\theta_{k+1} \gets \theta_k + \tau \delta\theta$
        \EndFor
        \end{algorithmic}
    \textbf{Output:} final map $T_{\theta_K}$
    \label{alg:trad-train}
  \end{algorithm}
\end{minipage}
\hfill
\begin{minipage}[t]{.48\linewidth}
    \begin{algorithm}[H]
    \flushleft
    \caption{Constrained gradient flow}
    \vspace{1mm}
    \textbf{Inputs:} source measure $\rho_0 \in \PRd$, initial parameter $\theta_0 \in \bR^m$, step size $\tau > 0$ \\
    \begin{algorithmic}[0]
        \For{$k\in[\![0,K-1]\!]$}
            \If{explicit scheme}
            \State find the minimizer $\theta^\star$ of \cref{eq:explicit_scheme}
            \ElsIf{implicit scheme}
            \State find the minimizer $\theta^\star$ of \cref{eq:implicit_scheme}
            \EndIf
            \State $\theta_{k+1} \gets \theta^\star$
        \EndFor
        \end{algorithmic}
    \textbf{Output:} final map $T_{\theta_K}$
    \label{alg:new-train}
  \end{algorithm}
\end{minipage}
\end{figure}

\subsubsection{Practical implementation details.} 
\label{sec:practical}
In most practical cases (including our numerical illustrations), the source measure $\rho_0$ can be sampled from, or is accessible through an empirical counterpart. 
One can thus approximate all integrals with respect to $\rho_0$ using their empirical counterpart based on i.i.d.~samples $(x_1,\dots,x_n)$ from $\rho_0$. We make those approximations explicit below.
\begin{itemize}[leftmargin=7mm]
    \item \emph{\cref{alg:trad-train}.}
    Using the chain rule, the gradient of the loss function that needs to be computed can be expressed as 
    \begin{equation} \label{eq:grad_rel_entropy_param}
        \nabla_\theta H (T_{\theta_k}{}_*\rho_0 \midd \gamma)
        = \int_{\bR^d}\big[\nabla\log(T_{\theta}{_*\rho_0})\circ T_{\theta}+\nabla V\circ T_{\theta} \big]^\top \nabla_\theta T_{\theta} \dd\rho_0,
    \end{equation}
    which is thus approximated by its empirical counterpart 
    \begin{equation} \label{eq:grad_rel_entropy_param_empiric}
        \frac{1}{n} \sum_{i=1}^n \big[\widehat{s}_\theta\big(T_\theta(x_i)\big) + \nabla V\big(T_\theta(x_i)\big)\big]^\top \nabla_\theta T_\theta(x_i),
    \end{equation}
    where $\widehat{s}_\theta$ is an estimation of $\nabla \log(T_{\theta{} *} \rho_0)$ based on the samples $(T_\theta(x_1),\dots,T_\theta(x_n))$ (see \cref{rem:score_estim} for details).
    \item \emph{\cref{alg:new-train}, explicit scheme.}
    To implement \cref{eq:explicit_scheme},
    we use its empirical counterpart 
        \begin{equation}\label{eq:practical_guided_training}
            \argmin_{\theta\in\Theta}\ \sum_{i=1}^n \Big\| -\big[\widehat{s}_{\theta_k}\big(T_{\theta_k}(x_i)\big)+\nabla V \big(T_{\theta_k}(x_i)\big) \big]- \frac{T_\theta(x_i) - T_{\theta_k}(x_i)}{\tau} \Big\|^2.
        \end{equation}
    This subroutine minimization procedure can be handled in a straightforward way by automatic differentiation as it only depends on $\theta$ through the term $\theta \mapsto T_\theta(x_i) = \nabla \phi_\theta(x_i)$ where $\phi_\theta$ is a neural network. To do so, one may use any optimization procedure (e.g., standard gradient descent, \textsc{adam}, ...). 
    
    \item \emph{\cref{alg:new-train}, implicit scheme.}
    To implement \cref{eq:implicit_scheme},
    we need to compute the gradient 
    \begin{equation}
        \nabla_\theta\Big( H(T_\theta{}_* \rho_0 \midd \gamma) + \frac{1}{2\tau} \| T_\theta - T_{\theta_k}\|^2_{\Lrhoz}\Big).
    \end{equation}
    While the first term requires to be manually computed using \eqref{eq:grad_rel_entropy_param_empiric}, the second term is directly handled using automatic differentiation on the empirical estimate 
    \begin{equation}
        \frac{1}{n} \sum_{i=1}^n \|T_\theta(x_i) - T_{\theta_k}(x_i)\|^2.
    \end{equation}
    Once this is done, one can use the resulting gradient as the input of any optimization procedure (e.g., standard gradient descent, \textsc{adam}, ...). 
\end{itemize}
Also, note that at each time step $k$ in \cref{alg:trad-train,alg:new-train}, we draw new samples $(x_1,\dots x_n)$ to estimate $\rho_0$. The same applies to the solving of the subroutines \cref{eq:explicit_scheme,eq:implicit_scheme}, where new samples are drawn at each time step of the optimization procedure.
\begin{remark}[Estimating the score function] \label{rem:score_estim}
    Both \cref{alg:trad-train,alg:new-train} require to estimate the \emph{score function} $x \mapsto \nabla \log (T_\theta{}_* \rho_0)(x)$ from the samples $(T_\theta(x_1),\dots,T_\theta(x_n))$. 
    We do so by relying on the self-entropic OT potential \cite{mordant2024entropic} of $\widehat{\rho}_0 = \frac{1}{n} \sum_{i=1}^n \delta_{x_i}$. We set the entropic regularization parameter to $5\%$ of the median squared distance in the point cloud $(T_\theta(x_i))_i$, which is an empirical choice used in \texttt{jax-ott} \cite{cuturi2022optimal} and which yields a reasonable behavior in practice. 
    More sophisticated ways of approximating the score could be considered; for instance, relying on denoising diffusion probabilistic models (DDPMs) \cite{ho2020denoising,song2020score}. 
    While these approaches are likely to perform better in complex methods, they rely on a parameterization of the score (typically by another neural network) that would impede our understanding of the numerical behavior of the flow. 
    In our numerical experiments, we therefore prefer to use a simple (yet reasonable) method in order to factor out, as much as possible, the difficult question of parameterizing an estimator of the score function. 
\end{remark}

\begin{remark}[MMD, Sinkhorn divergence, and drifting models]
    When choosing an MMD or the Sinkhorn divergence \cite{feydy2019interpolating} for the functional $D:\PRd\to\bR$, the score (and gradient of the potential) is replaced by the Wasserstein gradient of the corresponding functional. In this case, \cref{eq:practical_guided_training} is similar (up to renormalization factors in the MMD case) to the training dynamic of drifting models \cite{deng2026generative,he2026sinkhorn} on the class of ICNNs. Only few global convergence results for the Wasserstein gradient flows of MMDs are known \cite{boufadene2025global,chizat2026mmd}, and the question is still open for the Sinkhorn divergence (see \cite[Section~4.2]{carlier2024displacement} and \cite{hardion2026wasserstein} for the case of Gaussian measures), which impedes getting theoretical guarantees supporting the convergence of the numerical schemes.
\end{remark}

\subsubsection{Numerical illustrations}

In order to showcase the efficiency of the discretizations of the constrained gradient flow (\cref{alg:new-train}), we compare it to the standard Euclidean gradient descent (\cref{alg:trad-train}) and propose the following proof-of-concept experiment. 

\noindent $(i)$ \textit{Source and target measures.}
The target measure is $\gamma = N(0,I)$, or equivalently, $\gamma \propto e^{-V}$ with $V(x) = \smash{\frac{\|x\|^2}{2}}$.
The source measure $\rho_0$ is a Gaussian mixture with $4$ modes.
Both measures $\rho_0$ and $\gamma$ are sampled with $n=100$ atoms.
See \cref{fig:res} for a visual illustration.
\\
\noindent $(ii)$ \textit{Parameterization of $\K$.}
The parameterization of the set $\K$ of OT maps is $\theta\mapsto T_\theta\coloneqq \nabla \phi_\theta$, where $\phi_\theta$ is a simple ICNN with two hidden layers with 20 units each. We let $\Theta \subset \bR^m$ denote the set of possible parameterizations, with $m = 541$. 

\begin{wrapfigure}{r}{0.45\textwidth}
        \vspace{-1mm}
    \begin{subfigure}[t]{\linewidth}
        \flushright
        \caption{\footnotesize Constrained flow, implicit (ours)}
        \vspace{-1mm}
        \includegraphics[width=\linewidth]{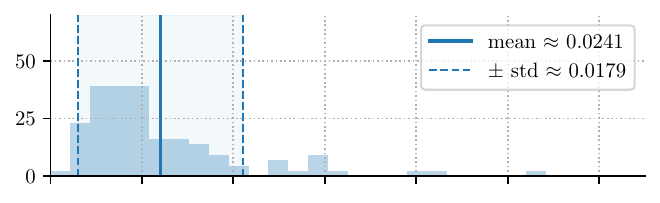}
    \end{subfigure}
    \begin{subfigure}[t]{\linewidth}
        \flushright
        \vspace{-2mm}
        \caption{\footnotesize Constrained flow, explicit (ours)}
        \vspace{-1mm}
        \includegraphics[width=\linewidth]{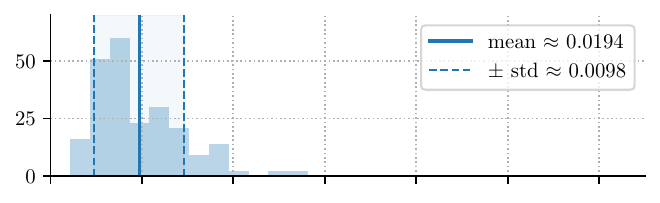}
    \end{subfigure}
    \begin{subfigure}[t]{\linewidth}
        \flushright
        \vspace{-2mm}
        \caption{\footnotesize Euclidean gradient descent}
        \vspace{-1mm}
        \includegraphics[width=\linewidth]{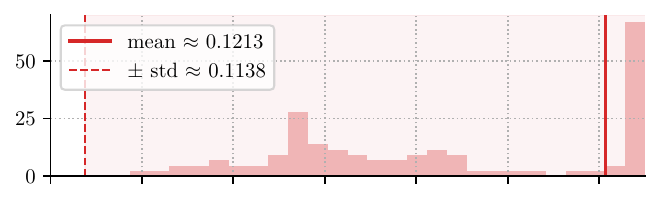}
    \end{subfigure}
    \begin{subfigure}[t]{\linewidth}
        \flushright
        \vspace{-2mm}
        \caption{\footnotesize Euclidean gradient, with \textsc{adam} optimizer}
        \vspace{-1mm}
        \includegraphics[width=\linewidth]{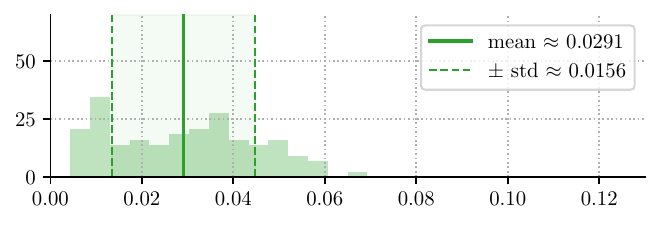}
    \end{subfigure}
  \caption{Histograms of the MMD values that result from methods (\textsc{a,\,b,\,c,\,d}). For each one, we plot and display the values of the mean and standard deviation over the 100 seeds. For method (\textsc{c}), values are clipped at 0.13 for the clarity of display (the mean and standard variation are kept unchanged).} 
  \label{fig:results_hist}
\end{wrapfigure}

\noindent $(iii)$ \textit{Methods to be compared and their parameters.}
We consider the following methods to be compared. First, our two discretizations of the constrained gradient flow:
\begin{enumerate}
    \item[{(\textsc{a})}] \cref{alg:new-train} with implicit scheme \cref{eq:implicit_scheme},
    \item[{(\textsc{b})}] \cref{alg:new-train} with explicit scheme \cref{eq:explicit_scheme},
\end{enumerate}
as well as the standard approach
\begin{enumerate}
    \item[{(\textsc{c})}] \cref{alg:trad-train}, the Euclidean gradient descent \cref{eq:eucl-gd}, 
\end{enumerate}
and, for the sake of completeness,
\begin{enumerate}
    \item[{(\textsc{d})}] \label{item:adam} \cref{alg:trad-train}, \emph{using \textsc{adam} for the optimization};
    yet, we stress that this is not a time discretization of \cref{eq:eucl-gf}, nor a discretization of a gradient flow in general \cite{barakat2021convergence}. 
\end{enumerate}
All four methods depend on the step size $\tau$, and on the number $K$ of iterates on $\theta$ (see \cref{alg:trad-train,alg:new-train}).
Methods (\textsc{a}) and (\textsc{b}) also depend on the number $K'$ of iterates in the minimization subroutines \cref{eq:implicit_scheme,eq:explicit_scheme}, respectively, for which we use the \textsc{adam} optimizer.
The values $\tau = 0.4$ for (\textsc{a,\,b}) and $\tau = 0.05$ for (\textsc{d}) seem to be in favor of their respective methods\footnote{Though, unsurprisingly, all those parameters are sensitive to each other.};
for (\textsc{a,\,b}), we let $K = 10$ and $K' = 100$, while we let $K=1\,000$ for (\textsc{d}), making the comparison between them fair in terms of number of calls to automatic differentiation.
The explicit scheme (\textsc{c}) appears to be numerically quite unstable for large values of $\tau$, and we had to choose a smaller step size $\tau = 0.001$ and do $K = 3\,000$ steps to reach convergence. 
\\
\noindent $(iv)$ \textit{Performance metric.}
The quality of an output $T_{\theta_K}$ is estimated by evaluating the MMD \cref{eq:MMD} with energy distance kernel $k(x,y)=-\|x-y\|$ between $T_{\theta_K{} *} \rho_0$ and $\gamma$, where this time $\rho_0$ and $\gamma$ are sampled with $n_{\text{eval}} = 10\,000$ atoms.
We run the experiment for 100 different seeds, where all methods share the same seed (i.e., start from the same parameter $\theta_0\in\Theta$, use the same samples from $\rho_0$, etc.), and report in \cref{fig:results_hist} the histogram of the values of the MMD as well as their mean and standard deviation.

\noindent
From \cref{fig:results_hist}, one can then make the following observations: 
\begin{itemize}[leftmargin=7mm]
    \item A fairly low MMD can be reached, suggesting that the parameterization $\theta\mapsto T_\theta$ of $\K$ we consider is sufficiently expressive to provide reasonable approximations of the actual OT map between $\rho_0$ and $\gamma$.
    A MMD value of $\approx 0.02$, though higher than the typical distance between two samples of size $10^4$ from $\cN(0,I)$ (suggesting that the parametrization we consider is nonetheless not fully expressive), is visually satisfying, as it can be seen on \cref{fig:res}.
    \item The Euclidean gradient descent (\textsc{c}) has, by a large margin, the worst performance. The mean value of the MMD is $\approx 0.12$, which reflects the fact that the optimization procedure often gets stuck in (visually unsatisfying) local optima---note that a value of $\approx 0.07$ for the MMD is already unsatisfying, as it can be seen on \cref{fig:res}.
    \item Unsurprisingly, switching from standard gradient descent (\textsc{c}) to the \textsc{adam} optimizer (\textsc{d}) yields a substantially better behavior. 
    Yet, there are still a substantial proportion of runs that end up above the $\approx 0.05$ MMD value.
    \item Our approaches (\textsc{a},\,\textsc{b}) yield the best results, with a slight edge and more consistency for the explicit scheme (\textsc{b}). This suggests that the theoretical results derived in \cref{sec:theory} \emph{can} translate into practical algorithms---see \cref{rem!limit_in_practice} for a discussion on the matter. 
\end{itemize}
We eventually stress that the natural gradient descent manages to reach (or be close to) the global minimum in very few steps ($K=10$) in the space of parameters, showcasing the benefits of using an appropriate geometry to update the parameters.
We tackled here the solving of the minimization subroutines \cref{eq:explicit_scheme,eq:implicit_scheme} in a naive and straightforward way (using a gradient descent with $K'=100$ steps); 
if one had an oracle to solve them, \cref{alg:new-train} would be vastly superior, in terms of computational efficiency, to other approaches in the setting we consider ($100$ times more).

\begin{figure}[!h]
    \begin{subfigure}[t]{0.49\textwidth}
        \includegraphics[width=\textwidth]{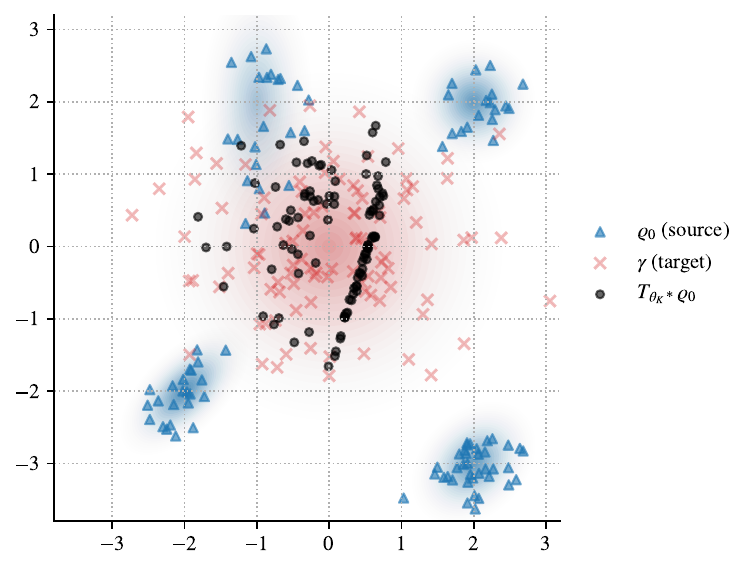}
    \label{fig:resa}
    \end{subfigure}
    \hfill
    \begin{subfigure}[t]{0.49\textwidth}
        \includegraphics[width=\textwidth]{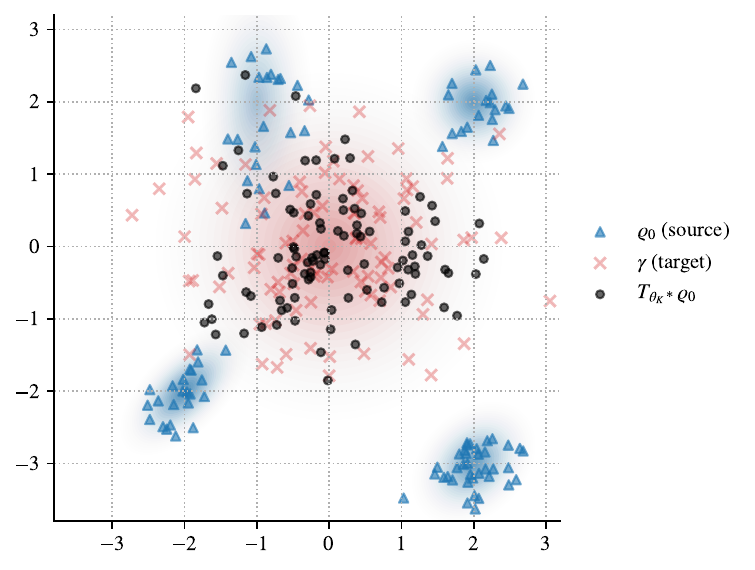}
    \label{fig:resb}
    \end{subfigure}
    \centering
    \vspace{-5mm}
    \caption{(\textit{left}) Example of an unsatisfying learned OT map, obtained with method (\textsc{d}), with a MMD value of $\approx 0.07$. (\textit{right}) Example of a satisfying learned OT map, obtained with method (\textsc{b}), with a MMD value of $\approx 0.02$.}
    \label{fig:res}
\end{figure}

\begin{remark}[From theoretical to numerical convergence guarantees] \label{rem!limit_in_practice}
    This numerical illustration aims at providing an elementary proof-of-concept to support the theoretical guarantees of \cref{sec:theory}: showing that the discretizations of the constrained gradient flow of a well-chosen functional (here, the relative entropy) on $\Theta$ \emph{can} help to reach better estimation of the actual OT map.
    \\
    There are naturally some gaps between the strong theoretical guarantees provided by \cref{thm:convergence} and \cref{cor:entropy} and practical set up we consider here. Namely, it could be that:
    \begin{enumerate}[$(i)$, leftmargin=*]
        \item the parameterization $\theta\mapsto T_\theta$ is not sufficiently expressive, that is, $\argmin_\theta H(T_\theta{}_* \rho_0 \midd \gamma)$ is far from the actual OT map $\optmap{\rho_\smallzero}{\gamma}$. This occurs for too restrictive classes of ICNNs (e.g., a single hidden layer with $10$ units seems to be unable to provide even a decent approximation of the actual OT map in our simple setting, no matter the optimization procedure).
        \item the subroutines \cref{eq:explicit_scheme,eq:implicit_scheme} are not solved exactly, and errors may accumulate over time. 
        \item we implement a gradient descent and not a gradient flow and rely on several estimates (relying on empirical samples, estimating the score, etc.) that can make the optimization scheme unstable.
    \end{enumerate}
    \cref{rem:errors} showed that for method (\textsc{b}), the accumulation of errors in the subroutines $(ii)$ and the time-discretization in point $(iii)$ do not impede the convergence of the scheme. Method (\textsc{a}) might be unstable (see \cref{rem:explicit-gd}), but went well in the setup we consider, possibly thanks to an implicit regularization induced by the neural networks.
    In closing, we believe that understanding whether those numerical schemes are reliable approximations of the constrained gradient flow \cref{eq:constrained-flow} when $n$ is large and when the neural network grows infinitely large is of interest, and left for future work. 
\end{remark}

Implementation details and code to reproduce the showcased experiment can be found in the public repository \url{https://github.com/theodumont/monge-constrained-flow}.

\addtocontents{toc}{\protect\setcounter{tocdepth}{0}}
\section*{Acknowledgements}
\addtocontents{toc}{\protect\setcounter{tocdepth}{1}}
TD wishes to thank Klas Modin and Guillaume Sérieys for precious technical discussions. 
This research is partly supported by the Bézout Labex, funded by ANR, reference ANR-10-LABX-58. 
TL is supported by the ANR project TheATRE ANR-24-CE23-7711.

% \newpage
\addtocontents{toc}{\protect\setcounter{tocdepth}{1}}
\renewcommand*{\bibfont}{\small}
\printbibliography
\newpage

\appendix
\phantomsection
\addcontentsline{toc}{section}{\textbf{Appendix}}

\addtocontents{toc}{\protect\setcounter{tocdepth}{0}}
\section{Additional definitions}
\label{appendix:omitted_notions}

In this section, we gather some additional definitions: generalized geodesics for measures that are not necessarily absolutely continuous (\cref{appendix:generalizedGeod}), cones in Hilbert spaces (\cref{sec:cones}), the divergence operator (\cref{sec:divergence}) and the \nameref*{thm:helmholtz} (\cref{app:helmholtz}).

\subsection{(Generalized) geodesics in the space of probability measures}\label{appendix:generalizedGeod}

In \cref{def:geod_cvx}, we considered geodesics in $\PRd$ in the special case where the source measure $\rho_0$ is absolutely continuous, and generalized geodesics in the special case where the anchor point measure $\bar\rho$ is absolutely continuous. Those notions are usually defined in the following more general setting.

Let $\rho_0,\rho_1\in\PRd$ and let $\pi\in\Pi_\text{o}(\rho_0,\rho_1)$ be an optimal transport plan. 
Then the curve
\begin{equation}
    \label{eq:geodesic}
    \rho_t=[(1-t)p_1+tp_2]_*\pi
\end{equation}
interpolates between $\rho_0$ and $\rho_1$. 
It can be shown to be a constant speed \emph{geodesic} in $\PRd$ (that is, $\W_2(\rho_s,\rho_t)=|t-s|\W_2(\rho_0,\rho_1)$ for all $0\leq s,t\leq1$), and all constant speed geodesics are of the form \cref{eq:geodesic} \cite[Theorem~7.2.2]{ambrosio2008gradient}.
A functional $D:\PRd\to\bR$ is said to be \emph{$\lambda$-convex along geodesics} for $\lambda \in \bR$ if for all $\rho_0,\rho_1\in \PRd$ there exists a curve $\rho_t$ of the form \cref{eq:geodesic} such that
\begin{equation}
    \label{eq:cvx-wass}
    D(\rho_t)\leq (1-t)D(\rho_0)+tD(\rho_1)-\frac\lambda2t(1-t)\W_2(\rho_0,\rho_1)^2.
\end{equation}
It is sometimes useful to require $D$ to be convex along more curves than the mere set of geodesics. 
Let $\bar\rho,\rho_0,\rho_1\in\PRd$. 
A \emph{generalized geodesic} between $\rho_0$ and $\rho_1$ with anchor point $\bar\rho$ is a curve of the form
\begin{equation}
    \label{eq:gen-geodesic}
    \rho_t=[(1-t)p_2+tp_3]_*\smash{\widetilde\pi},
\end{equation}
where $\smash{\widetilde\pi}\in\cP(\bR^d\times\bR^d\times\bR^d)$ has $\bar\rho,\rho_0$ and $\rho_1$ as first, second and third marginals, respectively, and such that $(p_1,p_2)_*\smash{\widetilde\pi}$ is an optimal transport plan between $\bar\rho$ and $\rho_0$ and $(p_1,p_3)_*\smash{\widetilde\pi}$ is an optimal transport plan between $\bar\rho$ and $\rho_1$.
This curve interpolates between $\rho_0$ and $\rho_1$.
The functional $D$ is said to be \emph{$\lambda$-convex along generalized geodesics} if for all $\bar\rho,\rho_1,\rho_2\in\PRd$, there exists a curve of the form \cref{eq:gen-geodesic} such that
\begin{equation}
    \label{eq:cvx-wass-gen}
    D(\rho_t)\leq (1-t)D(\rho_0)+tD(\rho_1)-\frac\lambda2t(1-t)\iiint_{\bR^d\times\bR^d\times\bR^d}\|y-z\|^2\dd\widetilde\pi(x,y,z).
\end{equation}
When choosing $\bar\rho=\rho_0$ in \cref{eq:gen-geodesic}, one recovers the geodesics \cref{eq:geodesic}; hence convexity along generalized geodesics is strictly stronger than mere convexity along geodesics.
If \cref{eq:cvx-wass} (resp.~\cref{eq:cvx-wass-gen}) holds for $\lambda=0$, we simply say that $D$ is \emph{convex} along geodesics (resp.~generalized geodesics).

\subsection{Cones in Hilbert spaces}
\label{sec:cones}
A (nonempty) subset $C$ of some Hilbert space $\cH$ is said to be a \emph{cone} if it is stable by the nonnegative scalings $v\mapsto \alpha v$ for all $\alpha\geq0$. The \emph{polar cone} of $C$ is defined as the cone
\begin{equation}
    C^*\coloneqq\{v\in \cH\mid\text{for all }w\in C,\, \langle v,w\rangle\leq0\}.
\end{equation}
If $C$ is convex, one has $C^{**}=\overline C$ and if $C$ is additionally closed, one has the \emph{Moreau decomposition} \cite{moreau1962decomposition}
\begin{equation}
    \label{eq:moreau}
    \text{for all }v\in\cH,\quad v=\proj_C(v)+\proj_{C^*}(v),
\end{equation}
where $\proj_C:v\mapsto \argmin_w\|w-v\|$ is the projection onto the nonempty closed convex set $C$ in the Hilbert space $\cH$.
This projection is characterized by the following equivalence \cite[Theorem~5.2]{brezis2011functional}
\begin{equation}
    u=\proj_C(v)\iff  \text{for all }w\in C,\ \langle v-u,w-u \rangle\leq0.
\end{equation}
Let $K$ be a convex subset of $\cH$ and let $x\in K$. The (Clarke) \emph{normal cone of $K$ at $x$} is the closed convex cone
\begin{equation}
    \label{eq:def-nor}
    \Nor_x\!K\coloneqq\{v\in \cH\mid \text{for all } y\in K,\,\langle v,y-x\rangle\leq0\},
\end{equation}
and the (Clarke) \emph{tangent cone of $K$ at $x$} is the closed convex cone
\begin{equation}
    \label{eq:def-tan}
    \Tan_x\!K\coloneqq\overline{\{v\in\cH \mid \text{there exists } t>0 \text{ such that } x+tv\in K\}}.
\end{equation}
See \cite[6.9 Theorem]{rockafellar1998variational}. Note that by convexity of $K$, this definition is equivalent to 
\begin{equation}
    \Tan_x\!K\coloneqq\overline{\{v\in\cH \mid \text{there exists } t_0>0 \text{ such that for all } t\leq t_0,\, x+tv\in K\}}.
\end{equation}
The normal and tangent cones are \emph{polar} one to each other, in the sense that $(\Nor_x\!K)^*=\Tan_x\!K$ and $(\Tan_x\!K)^*=\Nor_x\!K$. 

\begin{remark}[Tangent and normal cones in the nonconvex setting]
    If $K$ is not assumed to be convex, one can also define the Clarke tangent cone as
    \begin{equation}
        \Tan_x\!K\coloneqq \smash{\liminf_{\substack{K\ni y\to x\\t\to0}}} t^{-1}(K-y),
    \end{equation}
    or equivalently as 
    \begin{equation}
    \Tan_x\!K\coloneqq\bigg\{
      v\in\cH \ \bigg|\
      \begin{aligned}
      & \text{for all } (x_k)_k\subset K \text{ such that } x_k\to x, \text{ for all } t_k\to0,\\[-1mm]
      & \text{there exists } v_k\to v\text{ s.t.~}x_k+t_kv_k\in K\text{ for all }k\in\bN
      \end{aligned}
    \bigg\}
    \end{equation}
    and the Clarke normal cone $\Nor_x\!K$ as its polar cone (see \cite[6.2 Proposition]{rockafellar1998variational} or \cite[Definition 1.8]{mordukhovich2009variational}). Those definitions coincide with \cref{eq:def-nor,eq:def-tan} whenever $K$ is convex \cite[6.9~Theorem]{rockafellar1998variational}, which is the setting of this paper.
\end{remark}

\subsection{Divergence}
\label{sec:divergence}
Let $\fX_c$ denote the space of compactly-supported smooth vector fields on $\bR^d$ and $C^\infty_c(\bR^d,\bR)$ denote the space of compactly-supported smooth functions on $\bR^d$.
Let $\dd x$ denote the Lebesgue measure. The \emph{divergence operator} is defined as
\begin{equation}
    \div:\fX_c\to C^\infty_c(\bR^d,\bR)^*,\qquad \langle\div(v),f\rangle\coloneqq -\int_{\bR^d}df(v)\dd x.
\end{equation}
It is a bounded linear operator; by density of $\fX_c$ in $L^2(\bR^d,\bR^d)$, it therefore extends to $L^2(\bR^d,\bR^d)$. We use the same notation for the extended operator.
Let now $\rho_0\in\PRdac$ be a probability measure with a density with respect to the Lebesgue measure. Then the quantity $\div(\rho_0 v)$ is well-defined whenever $v$ belongs to $\LrhozRd$.
See, for instance, \cite[Definition 2.6]{gangbo2011differential}.

\subsection{Helmholtz--Hodge decomposition}
\label{app:helmholtz}
Let us recall the well-known Helmholtz--Hodge decomposition of vector fields (see, e.g., \cite[Box~6.2]{santambrogio2015optimal}).
\begin{theorem*}[Helmholtz--Hodge decomposition]
    \label{thm:helmholtz}
    Let $\Omega\subset\bR^d$ be a compact domain.
    Let $\rho_0\in\cP(\Omega)$ be a probability measure that has a density with respect to the Lebesgue measure.
    Then every vector field $v\in\Lrhoz(\Omega,\bR^d)$ can be decomposed into the sum of a gradient field $\nabla f$ and of a $\rho_0$-divergence-free vector field, that is,
    \begin{equation}
        \label{eq:helmholtz}
        v=\nabla f+w,\quad\text{and}\quad \div(\rho_0 w)=0,
    \end{equation}
    with $\nabla f,w\in \Lrhoz(\Omega,\bR^d)$.
    Assume additionally that $\rho_0>0$ almost everywhere.
    If one imposes Neumann boundary conditions for $w$, then this decomposition is unique, and the function $f$
    is the unique solution to the variational problem
    \begin{equation}
        \argmin_{f\in \dHrhoz(\Omega,\bR^d)}\int_\Omega \|v-\nabla f\|^2\dd\rho_0
    \end{equation}
    under the condition $\int_\Omega f=0$, as well as the unique solution to the elliptic equation
    \begin{equation}
        \div(\rho_0\nabla f)=\div(\rho_0 v).
    \end{equation}
\end{theorem*}

\subsection{The Wasserstein--Otto metric}
\label{app:wasserstein-metric}
The set $\PRd$ of probability measures can formally be seen as an infinite-dimensional Riemannian manifold, endowed with the following Riemannian metric:
\begin{equation}
    g^{\text{WO}}_\rho(\delta\rho,\delta\rho)=\int_{\bR^n} \|\nabla p\|^2\dd\rho,\qquad\text{where } \delta\rho=-\div(\rho\nabla p),
\end{equation}
see \cite{lafferty1988density,otto2001geometry,benamou2000computational}.
The induced distance on $\PRd$ is the Wasserstein distance \cref{eq:kantorovitch}.

\section{Omitted proofs and results}
\label{appendix:proofs}

In this section, we gather some proofs that were omitted in the main part of the paper: proofs of the expression \cref{eq:wass-gradient} of the Wasserstein gradient (\cref{prop:gradient-first-var}), of the first-order optimality condition \cref{eq:optimality-condition-div} for the constrained gradient flow (\cref{sec:optimality-condition}), and of the quadratic minimization formula \cref{eq:minimization-general} for the natural gradient (\cref{app:natural-gd}).
We also include some properties of lifted functionals (\cref{sec:lifting}) that were omitted during the paper.

\subsection{Wasserstein gradient and first variation}
We prove here the expression \cref{eq:wass-gradient} of the Wasserstein gradient.

\label{prop:gradient-first-var}
\begin{proposition}[The Wasserstein gradient is the gradient of the first variation]
    Let $D:\PRd\to\bR$. Assume that $D$ has a first variation $\smash{\frac{\delta D}{\delta\rho}}$ that is differentiable and that $D$ is regular, in the sense of \cite[Definition 10.1.4]{ambrosio2008gradient}. Then for all $\rho\in\PRd$,
    \begin{equation}
        \nablaa D(\rho)=\nabla \frac{\delta D}{\delta\rho}(\rho).
    \end{equation}
\end{proposition}
\begin{proof}
    Let us first assume that $\rho$ has a density with respect to the Lebesgue measure.
    Consider then the curve $\rho_\eps=(\id+\eps v)_*\rho$, for some $v=\nabla f\in\Tan_\rho\PRd$. Then for $\eps$ small enough, the map $(\id+\eps v)$ is the gradient of a convex function, hence an optimal transport map. By definition of the Wasserstein gradient of $D$ and by definition of the first variation of $D$,
    \begin{multline}
        \int_{\bR^d}\langle \nablaa D(\rho),v\rangle\dd\rho
        =\lim_{\eps\to0}\eps^{-1}\big(D(\rho_\eps)-D(\rho)\big)
        =\int_{\bR^d}\frac{\delta D}{\delta \rho}(\rho)\partial_\eps|_{\eps=0}\rho_\eps(x)
        \\
        =-\int_{\bR^d}\frac{\delta D}{\delta \rho}(\rho)\div(\rho v)
        =\int_{\bR^d}\langle\nabla\frac{\delta D}{\delta \rho}(\rho),v\rangle\dd\rho,
    \end{multline}
    since $\partial_\eps|_{\eps=0}\rho_\eps$ is given by $\partial_\eps|_{\eps=0}\rho_\eps=-\div(\rho v)$ and using an integration by parts. A continuity argument therefore shows that $\nablaa D(\rho)-\nabla\frac{\delta D}{\delta\rho}(\rho)$ is orthogonal to $\Tan_\rho\PRd$ and since it also belongs to $\Tan_\rho\PRd$, it is equal to zero, that is, $\nablaa D(\rho)=\nabla \frac{\delta D}{\delta\rho}(\rho)$, $\rho$-almost everywhere.
    In the case where $\rho$ does not have a density, one may approximate $\rho$ with absolutely continuous measures and invoke the regularity of $D$ to conclude.
\end{proof}

\subsection{First-order optimality condition}
\label{sec:optimality-condition}
We prove here the first-order optimality condition \cref{eq:optimality-condition-div}.
\begin{lemma}[First-order optimality condition for the constrained gradient flow]
    Let $T\in\K$, let $v\in\LrhoRd$ and let
\begin{equation}
    \bar w\coloneqq\argmin_{w\in \Tan_{T}\!\!\K}J(w),
    \qquad\text{where}\quad
    J(w)\coloneqq {\int_{\bR^d}}\|v\circ T-w\|^2\dd\rho_0.
\end{equation}
Assume that $\nabla(\dHrhozRd\cap C^2_c(\bR^d,\bR))\subset\Tan_{T}\!\K$. Then $\bar w$ satisfies
\begin{equation}
    \div(\rho_0\bar w)=\div(\rho_0v\circ T).
\end{equation}
\end{lemma}
\begin{proof}
    Taking variations $w_\eps\coloneqq\bar w+\eps\nabla\xi$ for $\xi\in\dHrhozRd\cap C_c^2(\bR^d,\bR)$, the optimality of $\bar w$ gives
    \begin{equation}
        0=\frac{\dd}{\dd \eps}J(\bar w+\eps\nabla \xi)\Big|_{\eps=0}=-2\int_{\bR^d}\langle v\circ T-\bar w,\nabla\xi\rangle\dd\rho_0=\int_{\bR^d}\xi\div(\rho_0(v\circ T-\bar w))\dd x.
    \end{equation}
    Since $\xi$ is arbitrary and by density of $C_c^2(\bR^d,\bR)$ in $\dHrhozRd$, one gets the result.
\end{proof}

\subsection{Natural gradient via quadratic minimization}
\label{app:natural-gd}
We prove here the quadratic minimization formula \cref{eq:minimization-general} for the natural gradient.

\begin{nblemma}{\ref*{prop:natural-gd}}[\naturalgdtitle]
\naturalgdall
    \begin{equation}
    \tag{\ref*{eq:minimization-general}}
    \label{eq:minimization-general-app}
        \argmin_{\delta\theta\in T_\theta\Theta}\,\big\|\!\grad^g_MF(\sigma_\theta)-d_\theta \sigma[\delta\theta]\big\|_{g}^2
    \end{equation}
    is unique and equal to $\grad^{\sigma^*g}_\Theta L(\theta)$.
\end{nblemma}
    \begin{proof}
    \label{sec:natural-gd-proof}
    Multiplying by $\frac12$ and expanding the squared norm gives that a solution $\delta\theta^\star$ of \cref{eq:minimization-general-app} also minimizes the quantity
    \begin{align}
        R(\delta\theta)
        \coloneqq&\,\frac12\big\|d_\theta \sigma[\delta\theta]\big\|_g^2-g_{\sigma_\theta}\big(\!\grad^g_MF(\sigma_\theta),d_\theta \sigma[\delta\theta]\big)\\
        =&\,\frac12\big\|\delta\theta\big\|_{\sigma^*g}^2-d_{\sigma_\theta}F\big[d_\theta \sigma[\delta\theta]\big]\\
        =&\,\frac12\big\|\delta\theta\big\|_{\sigma^*g}^2-d_{\theta}(F\circ \sigma) [\delta\theta],
    \end{align}
    where we used the definitions of the Riemannian gradient and of the pullback metric $\sigma^*g$.
    Differentiating with respect to $\delta\theta$ at optimality yields
    \begin{equation}
        0=(\sigma^*g)(\delta\theta^\star,\cdot)-d_{\theta}(F\circ \sigma)[\cdot],
    \end{equation}
    hence
        $\delta\theta^\star=\grad^{\sigma^*g}_\Theta L(\theta)$
    by definition of the Riemannian gradient again, and the proof is complete.
    \end{proof}

\subsection{Properties of lifted functionals}
\label{sec:lifting}

In this section, we provide several results on functionals that are lifted from $\PRd$ to $\LrhozRd$, which we will need to show the existence of solutions to the constrained gradient flow \cref{eq:constrained-flow} (\cref{thm:flow-defini}) and its convergence (\cref{thm:convergence}). Recall that $\rho_0\in\PRdac$ is an absolutely continuous probability measure and that $\pi:T\mapsto T_*\rho_0$ is the associated pushforward mapping.
Let $D:\PRd\to\bR$ be some functional on $\PRd$ and $F=D\circ\pi:\LrhozRd\to\bR$ be its lifted functional as defined in \cref{def:lifted}; to put it visually, $D$, $F$ and $\pi$ fit in the following diagram:
    \begin{center}
    \begin{tikzcd}
    \LrhozRd \vphantom{L}
    \ar[d, "\pi"]
    \ar[dr, "F"]
    \\
    \PRd \vphantom{L}
    \ar[r, "D"]
    &
    \bR.
    \end{tikzcd}
    \end{center}
Several properties of $D$ can be lifted to similar ones on $F$, and we make those links clear in the following lemmas. More precisely, we link the set of minimizers (\cref{lem:minimizers}), the convexity (\cref{lem:cvx-simple}), the lower semicontinuity (\cref{cor:lsc-topo}), and the differentiability (\cref{lem:frechet-diff}) of $F$ in $\LrhozRd$ to that of $D$ in $\PRd$.

\begin{lemma}[Minimizers of the lifted functional]
    \label{lem:minimizers}
    Consider $\rho_0$, $D$ and $F$ defined above.
    Then, minimizers of $D$ in $\PRd$ and minimizers of $F$ in $\LrhozRd$ and in $\K$ relate as follows:
    \begin{equation}
        \argmin_{\PRd} D=\pi(\argmin_{\LrhozRd} F)=\pi(\argmin_{\K} F),
    \end{equation}
    and $D$ has a unique minimizer in $\PRd$ if and only if $F$ has a unique minimizer in $\K$.
\end{lemma}
\begin{proof}
The result directly follows from the definition of $F$ as $D\circ\pi$ and from the bijectivity of the pushforward mapping $\K\ni T\mapsto T_*\rho_0\in\PRd$ given by \nameref{thm:brenier}.
\end{proof}
\noindent To guarantee the convergence of the constrained gradient flow \cref{eq:constrained-flow}, we need some convexity assumption on $F$ on the convex set $\K$ of optimal maps, subset of the Hilbert space $\LrhozRd$.
This convexity is easily rewritten in terms of convexity of $D$ along (generalized) geodesics in $\PRd$, as we show now.
\begin{lemma}[Convexity of the lifted functional]
\label{lem:cvx-simple}
    Consider $\rho_0$, $D$ and $F$ defined above and let $\gamma\in\PRd$.
Let us consider the following properties:
\begin{enumerate}[(i),leftmargin=1.3cm]
    \item[$(F2)$]
    $F$ is convex on $\K$, that is, along curves of the form 
    \begin{equation}
        (1-t)T_1+tT_2
        \qquad\text{for all } T_1,T_2\in \K.\qquad\quad\ \qquad\   
    \end{equation}
    \item[$(F3)$]
    $F$ is star-convex on $\K$ around $\optmap{\rho_\smallzero}{\gamma}$, that is, convex along curves of the form 
    \begin{equation}
        (1-t)T +t\optmap{\rho_\smallzero}{\gamma}
        \qquad\text{for all } T\in \K.\qquad\qquad\quad\qquad\ 
    \end{equation}
    \item[$(D2)$]
    $D$ is convex on $\PRd$ along generalized geodesics with anchor point $\rho_0$, that is, along curves of the form
        \begin{equation}
            [(1-t)\optmap{\rho_\smallzero}{\rho_\smallone}+t\optmap{\rho_\smallzero}{\rho_\smalltwo}]_*\rho_0\qquad \text{for all } \rho_1,\rho_2\in\PRd.\qquad\qquad\quad\ \qquad\ 
        \end{equation}
    \item[$(D3)$]
    $D$ is convex on $\PRd$ along generalized geodesics with anchor point $\rho_0$ and endpoint $\gamma$, that is, along curves of the form
        \begin{equation}
            [(1-t)\optmap{\rho_\smallzero}{\rho_\smallone}+t\optmap{\rho_\smallzero}{\gamma}]_*\rho_0\qquad \text{for all } \rho_1\in\PRd.\qquad\qquad\qquad\ \qquad\ 
        \end{equation}
\end{enumerate}
\noindent 
Then, properties \textit{$(F1)$} to \textit{$(D2)$} fit in the following diagram of implications:
    \begin{center}
    \begin{tikzcd}[row sep=normal, column sep=scriptsize]
    (F1) \ar[r, Rightarrow]
    &
    (F2) 
    \\
    (D1) \ar[u, Rightarrow]\ar[r, Rightarrow]
    &
    (D2)\ar[u, Rightarrow]
    \end{tikzcd}
    \end{center}
and the same holds when replacing ``convex'' by ``$\lambda$-convex'' for any $\lambda\in\bR$ above.
\end{lemma}

\begin{proof}
    First, since $\optmap{\rho_\smallzero}{\gamma}\in\K$, one has $(F1)\Rightarrow(F2)$; and taking $\rho_2\coloneqq \gamma$ in $(D2)$ yields $(D1)\Rightarrow(D2)$.\\
    Remark that by definition, $F=D\circ\pi$ is convex along some curve $T_t$ if and only if $D$ is convex along the curve $\pi(T_t)=T_t{}_*\rho_0$.
    Suppose now that $(D1)$ is true.
    Let $T_1,T_2\in\K$ and note $\rho_1\coloneqq T_1{}_*\rho_0$ and $\rho_2\coloneqq T_2{}_*\rho_0$. Then $T_1=\optmap{\rho_\smallzero}{\rho_\smallone}$ and $T_2=\optmap{\rho_\smallzero}{\rho_\smalltwo}$ and $D$ is therefore convex along the curve $[(1-t)T_1+tT_2]_*\rho_0$. Hence $F$ is convex along $(1-t)T_1+tT_2$; this proves $(F1)$.
    Suppose now that $(D2)$ is true. Let $T\in\K$ and let $\rho_1\coloneqq T_*\rho_0$. Then $T=\optmap{\rho_\smallzero  }{\rho_\smallone}$ and $D$ is therefore convex along the curve $[(1-t)T+t\optmap{\rho_\smallzero}{\gamma}]_*\rho_0$. Hence $F$ is convex along $(1-t)T+t\optmap{\rho_\smallzero}{\gamma}$; this proves $(F2)$.
\end{proof}

\noindent The next lemma then shows that the lower semicontinuity of the lifted functional $F$ is also inherited from the lower semicontinuity of $D$.

\begin{lemma}[Lower semicontinuity of the lifted functional]
    \label{cor:lsc-topo}
    Consider $D$ and $F$ defined above.
        If $D$ is weak-l.s.c.~on $\PRd$, then $F$ is strong-l.s.c.~on $\LrhozRd$.
        If additionally $D$ is convex along generalized geodesics with anchor point $\rho_0$, then $F$ is weak-l.s.c.~on $\K$.
\end{lemma}
\begin{proof}
From the inequality $\W_2(T_*\rho_0,S_*\rho_0)\leq \|T-S\|_{\Lrhoz}$ for all $T,S\in\LrhozRd$, one gets that the pushforward mapping $\pi:T\mapsto T_*\rho_0$ is continuous from $\LrhozRd$ with the strong topology to $\PRd$ with the weak topology. The first result then follows by composition. If additionally $D$ is convex along generalized geodesics with anchor point $\rho_0$, then $F$ is convex on $\K$ (see \cref{lem:cvx-simple}) and the second result follows from the fact that lower semicontinuity for the strong and weak topologies coincide for convex functionals in Hilbert spaces \cite[Corollary~3.9]{brezis2011functional}.
\end{proof}
\noindent Finally, the following result from \citeauthor{gangbo2019differentiability}~\cite[Corollary~3.22]{gangbo2019differentiability} allows to link the differentiability properties of $F$ in $\LrhozRd$ to that of $D$ in $\PRd$.
\begin{lemma}[Differentiability of the lifted functional]
    \label{lem:frechet-diff}
    Consider $D$ and $F$ defined above.
    For all $T\in\LrhozRd$, $F$ is (Fréchet) differentiable at $T$ if and only if $D$ is (Wasserstein) differentiable at $\pi(T)=T_*\rho_0$, and in this case
    \begin{equation}
        \nabla F(T)=\nablaa D(T_*\rho_0)\circ T,
    \end{equation}
    where the equality is to be understood $\rho_0$-a.e.
\end{lemma}

\end{document}